\documentclass[a4paper]{amsart}

\usepackage{amsfonts, amsmath, amssymb, amsthm, epsfig}
\usepackage[british]{babel}
\usepackage{todonotes}
\usepackage{standalone,subcaption}
\usepackage[pagebackref=true, hidelinks]{hyperref}
\usepackage[capitalize]{cleveref}
\usepackage[framemethod=TikZ]{mdframed}
\usepackage[alphabetic, nobysame, backrefs]{amsrefs}

\usetikzlibrary{calc}

\usepackage{array}
\newcolumntype{C}{>{$}c<{$}} %
	\newcolumntype{R}{>{$}r<{$}} %

\usetikzlibrary{decorations.markings}
\usetikzlibrary{arrows}
\usetikzlibrary{arrows.meta, positioning}
\usetikzlibrary{patterns}

\calclayout

\input{utils/custom-commands}

\author{Giovanni Italiano}
\address{\parbox{\linewidth}{Mathematical Institute, University of Oxford, \\Andrew Wiles Building,
		Woodstock Road, OX2 6GG Oxford, UK}\vspace{1.5pt}}
\email{italiano at maths dot ox dot ac dot uk}

\author{Matteo Migliorini}
\address{Karlsruher Institut für Technologie, Englerstraße 2, 76131 Karlsruhe}
\email{matteo dot migliorini at kit dot edu}

\title{Perfect circle-valued Morse functions on hyperbolic $6$-manifolds}

\def\numfacets{c}
\def\polytopename{P}
\NewDocumentCommand{\pol}{o}{\polytopename\IfValueT{#1}{_{#1}}}
\def\dualpol{P^*}
\NewDocumentCommand{\facet}{o}{F\IfValueT{#1}{_{#1}}}
\NewDocumentCommand{\varfacet}{o}{F'\IfValueT{#1}{_{#1}}}
\NewDocumentCommand{\face}{o}{P'\IfValueT{#1}{_{[#1]}}}
\NewDocumentCommand{\dualface}{o}{P'^*\IfValueT{#1}{_{[#1]}}}
\def\mfld{M}
\def\truncmfld{\bar M}
\def\truncpolytopename{\bar \polytopename}
\NewDocumentCommand{\truncpol}{o}{\truncpolytopename\IfValueT{#1}{_{#1}}}
\def\extcubulation{\bar\cubulation}
\def\extf{\bar f}
\def\cubulation{C}
\def\coxetergrp{W}
\def\commutators{W'}
\NewDocumentCommand{\cube}{o}{Q\IfValueT{#1}{_{#1}}}
\NewDocumentCommand{\varcube}{o}{Q'\IfValueT{#1}{_{#1}}}
\NewDocumentCommand{\varvarcube}{o}{Q''\IfValueT{#1}{_{#1}}}
\NewDocumentCommand{\state}{o}{s\IfValueT{#1}{_{#1}}}
\NewDocumentCommand{\varstate}{o}{s'\IfValueT{#1}{_{#1}}}
\NewDocumentCommand{\status}{o m}{\state[#1](#2)}
\NewDocumentCommand{\inheritedstate}{o}{s\IfValueT{#1}{_{[#1]}}}
\NewDocumentCommand{\inheritedstatus}{o m}{\inheritedstate[#1](#2)}
\def\Out{\ensuremath{O}}
\def\In{\ensuremath{I}}
\def\stateout{\state_\Out}
\def\statein{\state_\In}
\def\varstateout{\varstate_\Out}

\def\move{m}
\def\deffacets{\mathfrak F_{\cube}}
\def\baryf{f}
\def\foneskel{f^{(1)}}

\def\abcover{\widetilde\mfld}

\NewDocumentCommand{\vectorspace}{s}{\IfBooleanTF#1{(\ZZ[2])}{\ZZ[2]} ^\numfacets}

\def\gosset{2_{21}}

\begin{document}

\begin{abstract}
	We build the first example of a hyperbolic $6$-manifold that admits a perfect circle-valued Morse function, which can be considered as the analogue of a fibration over the circle for manifolds with non-vanishing Euler characteristic.
	As a consequence, we obtain a new example of a subgroup of a hyperbolic group which is of type $\mathcal{F}_2$ but not $\mathcal{F}_3$.
\end{abstract}

\maketitle

\section{Introduction}

A peculiar phenomenon that has been widely studied in the previous decades is the existence of many hyperbolic $3$-manifolds that are also fibre bundles over the circle $S^1$. The first example of such a manifold was constructed by J\o rgensen \cite{JorgensenFibering} in 1977, and quickly captured the interest of many topologists, as some behaviours of these examples are quite paradoxical, as noted in \cite{CannonThurston}.

After Jørgensen's example, Thurston classified which diffeomorphisms of surfaces yield mapping tori that carry hyperbolic metrics \cite{Thurston3Manifolds}, providing a simple tool to construct many such fibring hyperbolic $3$-manifolds.
A celebrated theorem of Agol and Wise \cites{AgolVirtualHaken, WiseGroupStructure} shows how such a phenomenon is actually ubiquitous; every finite-volume hyperbolic $3$-manifold has a finite cover that fibres over $S^1$.

Despite the omnipresence of fibring manifolds in dimension $3$, the first progress in building similar examples in higher dimensions is much more recent. The first example in higher dimension is due to Martelli and the authors \cite{IMMHyperbolic5Manifold}, and consists of a $5$-manifold tessellated by some copies of a particular right-angled polytope.

In this paper, we look at dimension $6$. Hyperbolic $6$-manifolds cannot fibre due to an Euler characteristic obstruction; the natural generalisation is the notion of \emph{perfect circle-valued Morse function}. We prove the following.

\begin{maintheorem}\label{thm:M6}
	There exists a hyperbolic $6$-manifold $M^6$ that admits a perfect circle-valued Morse function.
\end{maintheorem}

To prove the theorem, we construct an explicit manifold $M^6$ equipped with a Morse function onto the circle whose critical points are all of index $3$. The manifold $M^6$ is cusped, and as was the case for the fibring $5$-manifold $M^5$ in \cite{IMMHyperbolic5Manifold}, the techniques used in this paper cannot produce a closed manifold in dimension $ n \geq 5 $ due to the lack of compact right-angled hyperbolic polytopes \cite{PotyagailoVinberg}.

As an analogue of \cite{BattistaHyperbolic4manifolds}*{Corollary 3}, we get the following.
\begin{maincorollary}
    There are infinitely many finite-volume (cusped) hyperbolic $6$-manifolds $M$ admitting a handle decomposition with bounded numbers of $1$-, $2$-, $4$- and $5$-handles. In particular, there are infinitely many $M$ with bounded Betti numbers $b_1(M)$, $b_2(M)$, $b_4(M)$, $b_5(M)$ and bounded rank of $\pi_1(M)$.
\end{maincorollary}

\subsection*{Perfect Morse functions}
Let $M$ be an even-dimensional hyperbolic manifold. Given a Morse function $f \colon M \to S^1$, the alternating sum of the number of critical points of each index yields the Euler characteristic $\chi(M)$. It follows that the number of critical points for such a function is at least $\abs{\chi(M)}$. In particular, a necessary requirement for $f$ to be a fibration is that the Euler characteristic of $M$ vanishes.

However, the Chern-Gauss-Bonnet formula implies that the Euler characteristic of $M$ is proportional to its volume and never zero; therefore $M$ cannot fibre over the circle. The natural generalisation of the fibring problem in this setting is whether $M$ admits a Morse function with a minimal number of critical points.

\begin{definition}
	A Morse function $f\colon M \to S^1$ is said to be \newterm{perfect} if it has exactly $\abs{\chi(M)}$ critical points.
\end{definition}
In the odd-dimensional case $\chi(M)=0$, so perfect Morse functions are actually fibrations.
Perfect circle-valued Morse functions have already been studied in this context by Battista and Martelli \cite{BattistaHyperbolic4manifolds}. In particular, they constructed several hyperbolic $4$-manifolds, both closed and cusped, that admit such a function.

\subsection*{Finiteness properties of groups}
The manifold $M^6$ can be used to construct subgroups of hyperbolic groups with exotic finiteness properties.

Recall that a group is said to be of type \finitetype[n] if it admits a classifying space with finite $n$-skeleton, and of type $ \finitetype $ if the classifying space has finitely many cells. This generalizes classical notions of finiteness, as being of type \finitetype[1] is equivalent to being finitely generated, and \finitetype[2] to finitely presented. Hyperbolic groups are of type \finitetype[n] for every $n$.

A natural question is under which conditions a subgroup of a hyperbolic group is necessarily hyperbolic itself. Rips \cite{RipsF1notF2} in 1982 presented a construction that produces many examples of finitely generated subgroups of hyperbolic groups that are not finitely presented, and therefore not hyperbolic; finding finitely presented counterexamples turns out to be a much harder task. Except for the group built in \cite{IMMHyperbolic5Manifold}, which is of type $\finitetype$, all the known examples rely on the failure of some higher finiteness property to obstruct hyperbolicity. %

A strategy for constructing a subgroup of a hyperbolic that is finitely presented but not \finitetype[3] was first proposed by Gromov \cite{GromovHyperbolicGroups}. In fact, Bestvina noticed that the ambient group could not be hyperbolic.

Later Brady \cite{BradyBranchedCoverings}, building on Gromov's idea, managed to construct a working example. This remained the only known finitely presented, non-hyperbolic subgroup of a hyperbolic group until much more recently, with constructions by Lodha \cite{LodhaF2notF3} and Kropholler \cite{KrophollerF2notF3}.
An example of a subgroup of type $\finitetype[3]$ not $ \finitetype[4] $ was found by Llosa Isenrich, Martelli and Py \cite{LlosaPyMartelli} starting from the fundamental group of the hyperbolic $8$-manifold constructed in \cite{IMMAlgebraicFibering}; more generally, Llosa Isenrich and Py \cite{LlosaPy} constructed subgroups that are type $ \finitetype[n] $ not $ \finitetype[n+1] $, using methods from complex hyperbolic geometry.

As consequence of \cref{thm:M6}, we construct a finitely presented subgroup of a hyperbolic group that is not $ \finitetype[3] $, that adds to this now growing class of examples.
To the best of our knowledge this is the first example of this class that is produced using real hyperbolic geometry.

\subsection*{Main techniques}
Jankiewicz, Norin, and Wise \cite{JankiewiczNorinWise} described an algorithm, starting from a so-called \newterm{system of moves} and an \newterm{initial state}, that produces a cube complex $ \cubulation $ with a Bestvina-Brady Morse function $ f \colon \cubulation \to S^1 $.
When certain combinatorial conditions are satisfied, then $f$ has connected ascending and descending links and by \cite{BestvinaBradyMorseTheory} it is an algebraic fibring.

Getting a topological fibration, or more generally a perfect Morse function, is a much harder task. A fibration of a hyperbolic $5$--manifold was eventually found in \cite{IMMHyperbolic5Manifold}, but the authors needed to relax the condition on the moves, defining a notion of \emph{set of moves} that allows some degeneracy on certain codimension--$2$ faces, called \newterm{bad ridges}. This degeneracy was controlled by choosing a suitable subdivision of the cube complex, and it allowed to construct a function with \emph{collapsible} ascending and descending links.

In this paper, we generalise this by introducing the notions of \emph{bad} and \emph{good} faces of a polytope (see \cref{def:good-polytope}), that depend only on the set of moves.
Since bad faces can now appear in every codimension, we are only able to define the Bestvina-Brady Morse function on the barycentric subdivision of $ \cubulation $.
This has the drawback that we introduce many more vertices for which we have to check ascending and descending links. We solve this issue with \cref{thm:inherited-state-collapsibility}, where we give a combinatorial condition to check their collapsibility.

We then apply this general strategy to the polytope $P^6$, which belongs to a remarkable family of polytopes introduced in \cite{AgolLongReid} and studied in \cite{PotyagailoVinberg}. Since we are in even dimension, we cannot get a fibration, so some critical points are expected. Indeed, we get some ascending and descending links that are not collapsible, but they collapse to a sphere. After taking care of the cusps, a smoothing argument leads to the proof of \cref{thm:M6}.

\subsection*{Structure of the paper}

The paper is structured as follows.

\begin{itemize}
	\item In \cref{sec:preliminaries} we recall the construction described in \cite{JankiewiczNorinWise}, while introducing \emph{good} and \emph{bad} faces, and the notion of inherited state.
	\item In \cref{sec:bary-sub} we show how to construct a Bestvina-Brady Morse function $ \baryf \colon \sd \cubulation \to S^1 $ on the barycentric subdivision $ \sd C $ of the cubulation $ \cubulation $ associated with a hyperbolic polytope $P$, such that the ascending and descending links of $\baryf$ are collapsible at the barycentres of good faces, and are controlled by the inherited state at the barycentres of bad ones.
	\item In \cref{sec:P6} we describe the combinatorics of the polytope $P^6$, while focusing on a particular subgroup of its symmetries, and we choose a set of moves and an initial state. By using the results in \cref{sec:bary-sub} we construct a Bestvina-Brady Morse function $f$ with ascending and descending links that are either collapsible, or collapse to a $2$-sphere.
	\item In \cref{sec:cusps} we embed $C$ into another cubulation $ \extcubulation$ whose interior is diffeomorphic to $ \mfld $, and we extend $f$ to $ \extf \colon  \extcubulation \to S^1 $ while preserving the condition on ascending and descending links.
	\item In \cref{sec:smoothing} we show that we can smoothen $ \extf $ to a smooth perfect circle-valued Morse function, thus proving \cref{thm:M6}.
	\item In \cref{sec:finiteness} we show how to use \cref{thm:M6} to construct a finitely presented subgroup of a hyperbolic group that is not of type $ \finitetype[3] $.
	\item We conclude with some questions in \cref{sec:comments}.
\end{itemize}

\subsection*{Acknowledgments}
The authors would like to thank Bruno Martelli for introducing them to the problem and for many interesting and useful discussions on the topic, and Claudio Llosa Isenrich for directing them towards the application described in \cref{sec:finiteness}, and for his many helpful comments.

The first author gratefully acknowledges support from the Royal Society through the Newton International Fellowship (award number: NIF\textbackslash R1\textbackslash 231857).
The second author gratefully acknowledges funding by the DFG 281869850 (RTG 2229).
Both authors are members of the INdAM GNSAGA research group.

\section{Preliminaries}\label{sec:preliminaries}

Jankiewicz, Norin and Wise described in \cite{JankiewiczNorinWise} a combinatorial game that, given a right-angled Coxeter group equipped with some combinatorial data, produces a $\CAT 0$ cube complex $\cubulation$ equipped with a piecewise-linear map $ f\colon \cubulation \to S^1 $. %
The aim of this section is to recall the notions of Bestvina-Brady Morse theory on which the game is based on, and then present the game in a slightly modified form, which can be seen as a generalisation of the procedure employed in \cite{IMMHyperbolic5Manifold}.

\subsection{Bestvina-Brady Morse theory}

Given a cell complex $X$, and two cells $ \sigma, \tau $ of $X$, we say that $ \sigma $ is a \newterm{face} of $ \tau $, or equivalently that $ \tau $ is a \newterm{coface} of $\sigma$, if $ \sigma \subseteq \tau $. We add the adjective \newterm{proper} if the inclusion is strict.

The following definitions are taken from \cite{BestvinaBradyMorseTheory}.
\begin{definition}
	An \newterm{affine cell complex} is a cell complex where every cell $ \sigma $ is equipped with a characteristic function $ \chi_\sigma \colon \sigma \to P_\sigma $, which is a homeomorphism to an affine polytope $ P_\sigma \subseteq \RR^d $, such that whenever $ \tau $ is a face of $ \sigma $, with inclusion map $ i \colon \tau \hookrightarrow \sigma $, then $ \chi_\sigma \circ i \circ \inv{\chi_\tau} $ is the restriction of an affine homeomorphism of $\RR^d$.
\end{definition}

Affine cell structure have a natural piecewise-linear structure, inherited by the affine structure on the cells.
Examples are simplicial complexes and cube complexes, where characteristic functions are identifications with the standard simplex or cube. Note that every subdivision of an affine cell complex is naturally an affine cell complex.

\begin{definition}
	Let $X$ be an affine cell complex, and let $ f\colon X \to S^1 $ be a map. We say that $f$ is a \newterm{circle-valued Bestvina-Brady Morse function} if for every positive-dimensional cell $ \sigma $ the restriction $ \restrict f \sigma $ is affine and nonconstant.
\end{definition}

In the definition above, \emph{affine} means that there is an affine map $F\colon \RR^d \to \RR $ such that $ f=\pi \circ F \circ \chi_\sigma $, where $ \pi\colon \RR \to S^1 $ is the standard covering map. Since we mostly work with circle-valued functions, we will drop the \emph{circle-valued} adjective, and use the term \newterm{real-valued Bestvina-Brady Morse function} when the codomain is $\RR$.

Bestvina-Brady Morse functions are the piecewise-linear analogue of their more widespread smooth counterpart. There is also a piecewise-linear notion that mimics the concept of regular and critical point; this is given by the \newterm{ascending and descending links}.

\begin{definition}
	Let $v$ be a vertex of an affine cell complex $X$, and let $ f \colon X \to S^1 $ be a Bestvina-Brady Morse function. The \newterm{ascending (descending) link} of $v$, denoted by $ \linkup[f] v $ (resp.~$ \linkdown[f] v $), is the subcomplex of $ \link v $ spanned by cells $ \sigma $ such that $v$ is the minimum (maximum) of any lift of $\restrict f \sigma  \colon \sigma \to S^1$ to a real-valued function.

	If the function $f$ is clear from the context, we just write $ \linkup v $ and $ \linkdown v $.
\end{definition}

Descending links control how the topology of sublevels of $f$ change when crossing a vertex, as explained in the following proposition.
We refer to \cite{RourkeSanderson} for the definition of \emph{collapsing}.

\begin{proposition}[\cite{BestvinaBradyMorseTheory}]\label{prop:bb-deformation-lemma}
	Suppose that $f\colon X \to \RR$ is a real-valued Bestvina-Brady Morse function on an affine cell complex $X$, and let $ a,b \in \RR $.

	If there are no vertices of $X$ with value in $(a, b]$, then $ X_{\leq b} $ collapses to $ X_{\leq a} $. If there is a single vertex $v$ with value in $ (a,b] $, then $X_{\leq b}$ collapses to $ X_{\leq a} $ with a cone attached over $ \linkdown v  $, where $ \linkdown v  $ embeds naturally inside $ f^{-1}(a) $.
\end{proposition}

We will see in \cref{sec:smoothing} that any vertex with collapsible (i.e.~that collapses to a point) descending link can be smoothened to a regular value for a smooth Morse function. On the other hand, a vertex with descending link that collapses to a $(k-1)$-sphere can be replaced with an index-$k$ critical point. This will be crucial for the proof of \cref{thm:M6}.

\subsection{Jankiewicz, Norin, and Wise's algorithm}

We start by fixing the terminology regarding the main objects of this paper, which are hyperbolic polytopes.

\begin{definition}
	A \newterm{hyperbolic $n$-polytope} $P$ is a finite volume intersection of finitely many hyperspaces of $ \HH^n $. A \newterm{face} of $P$ is either $P$ itself, or the non-empty intersection of $P$ with the boundary of a hyperspace containing it. A face of $P$ is \emph{proper} if it is different from $P$.

	We refer to faces of dimension $ n-1 $, $ n-2 $, $ 1 $, and $0$ as \newterm{facets}, \newterm{ridges}, \newterm{edges}, and \newterm{(finite) vertices} respectively. An \newterm{ideal vertex} is a point of $ \boundary\HH^n $ contained in the closure of $P$, and it is not considered a face of the polytope.
\end{definition}

\begin{remark}
	In contrast to the usual convention, we exclude the empty set from the faces of a polytope. On the other hand, it is convenient to consider the polytope as a face of itself.
\end{remark}

Let $\pol$ be a hyperbolic right-angled $n$-polytope, meaning that the dihedral angle between any two adjacent facets is $ \pi/2 $.
Let $\numfacets$ be the number of facets of $\pol$, denoted by $\facet[1], \dots, \facet[\numfacets]$ (after choosing an arbitrary ordering).

The group generated by the reflections along the faces of $\pol$ is a right-angled Coxeter group $\coxetergrp$ with $\numfacets$ generators. Its commutators subgroup $\commutators$ is a torsion-free subgroup of index $ 2^\numfacets $. 

So the quotient $ \quotient{\Hyp n}{\commutators} $ is a finite volume hyperbolic manifold $\mfld$, tessellated in $ 2^\numfacets $ copies of $\pol$. A concrete description of $\mfld$ is the following. For every $ v \in \ZZ[2]^\numfacets $ consider a copy of $\pol$ denoted by $ \pol[v] $; then we glue $ \pol[v] $ to $ \pol[v+e_i] $ along the facet $\facet[i]$ with the identity map, for all $v$ and for every $ i \in \range \numfacets$.
\begin{remark}
	The above construction, originally introduced by L\"obell \cite{Loebell}, coincides with the one described in \cites{BattistaHyperbolic4manifolds,IMMAlgebraicFibering,IMMHyperbolic5Manifold} in terms of colourings, in the particular case where we choose the ``inefficient colouring'' in which all facets have distinct colours. This is not restrictive, as the heavy work is done by the set of moves that we define below; changing the colouring yields virtually the same result.
\end{remark}

The dual of the tessellation is a cube complex $\cubulation$, since $\pol$ is right-angled; this can also be seen as the quotient of the Davis complex of $\coxetergrp$ by the action of $\commutators$ by left-multiplication.

When $\mfld$ is compact, then $\cubulation$ is homeomorphic to $\mfld$; otherwise, the cube complex $\cubulation$ is just a deformation retract of $\mfld$, also called a \newterm{spine}. We will see in \cref{sec:cusps} that $M$ admits a cube complex structure $ \extcubulation $ in which $\cubulation$ embeds nicely.

The aim is to define a map $ f\colon \mfld \to S^1 $, and we first define it on $\cubulation$. To do so, we need the following two combinatorial objects.

\begin{definition}
	A \newterm{set of moves} is any partition of the facets. The sets of the partition are called \newterm{moves}. If the facets belonging to each move are pairwise disjoint, the set of moves is called \newterm{sparse}.
\end{definition}

Similarly to \cite{IMMHyperbolic5Manifold}, we allow non-sparse set of moves, which is in contrast with the definition of \emph{system of moves} given in \cite{JankiewiczNorinWise}. This will be crucial for the concrete example we examine in this paper.

\begin{definition}
	A \newterm{state} $\state$ is any labelling of the facets of $P$ with $\In$ (for In) or $\Out$ (for Out). The label on a facet $\facet$ is called the \newterm{status} of $\facet$ and is denoted by $\status F$.
\end{definition}

A state induces a bipartition on the vertices of the polytope $\dualpol$ dual to $\pol$. Since $\pol$ is right-angled and therefore simple, then $\dualpol$ is a simplicial complex. We consider the full subcomplexes whose vertices are the vertices of $\dualpol$ with label $\Out$ and $\In$; denote these subcomplexes by $ \stateout, \statein $ respectively.

\begin{definition}
	A state is:
	\begin{itemize}
		\item \newterm{legal} if both $\stateout$ and $\statein$ are connected;
		\item \newterm{$m$-legal} if both $\stateout$ and $\statein$ are $m$-connected;
		\item \newterm{totally legal} if both $\stateout$ and $\statein$ are collapsible.
	\end{itemize}
\end{definition}

Moves act on states naturally: if $\move$ is a move, then $\move$ acts by inverting the status of all the facets belonging to $\move$, while preserving the status of all the others.

The objective is to define a state on all copies $\pol[v]$. To do so, we start by fixing a state $\state[0]$ on $\pol[0]$, called the \newterm{initial state}. We then define an action of $\vectorspace$ on the states, by saying that the element of the basis $e_i$ acts by applying on each state the move that contains the facet $\facet[i]$. Since the actions of moves commute with each other, this defines an action of the whole $\vectorspace$.
We use this action to propagate the initial state $\state[0]$ of the polytope $\pol[0]$ to all other polytopes of the tessellation; to do so, we assign to the polytope $\pol[v]$ the state $ \state[v] = v \cdot \state[0] $.

In particular, the state $ \state[v+e_i] $ is obtained from $ \state[v] $ by acting with the move containing $\facet[i]$. This can be summarized with the slogan:

\mdfsetup{skipabove=10pt,skipbelow=10pt,roundcorner=10pt}
\begin{mdframed}
	\begin{center}
		\textit{whenever we cross a facet $\facet$, \\ we invert the status of $\facet$ and all other facets that are in the same move as $\facet$}.
	\end{center}
\end{mdframed}

We use the states on the polytopes to define an orientation on each edge of the cube complex $\cubulation$ that is dual to the tessellation in copies of $P$.
Every vertex of the dual cube complex $\cubulation$ is dual to a polytope of the tessellation. We use the same letter $v$ for both the vectors of $\vectorspace$ and the vertices of $\cubulation$, since there is a natural correspondence between them. So for $ v \in \vectorspace $ we also denote by $v\in C$ the vertex dual to the polytope $\pol[v]$.

Let $\facet[i]$ be a facet, and let $e$ be the edge connecting $v$ with $v+e_i$. If the status $\status[v]{\facet[i]}$ is \Out, we orient $e$ from $v$ to $ v+e_i $; otherwise we put the opposite orientation.
Note that the status $\status[v]{\facet[i]}$ of the facet $\facet[i]$ inside $\pol[v]$ is always the opposite of the status $\status[v+e_i]{\facet[i]}$ of the facet $\facet[i]$ inside $\pol[v+e_i]$. This simple fact follows from how the state propagation law is defined, and guarantees that the orientation on every edge of $\cubulation$ is well-defined.

Denote with $ \skel1\cubulation $ the $1$-skeleton of $\cubulation$. We define a map $ \foneskel \colon \skel1\cubulation \to S^1 $ as follows: every oriented edge is canonically identified with the interval $[0,1]$, which we map to the circle via the usual quotient $ \quotient\RR\ZZ=S^1 $. We would like to extend this map to the whole $\cubulation$.

Note that choosing an orientation for every edge naturally produces a cellular $1$-cochain, which maps every oriented edge to either $ 1 $ or $-1$, depending on whether the orientations coincide. To admit a continuous extension to $\cubulation$, we need this $1$-cochain to be a cocycle, meaning that every square should have two edges oriented clockwise and two oriented counter-clockwise.

However, for our purpose we also need the map to be piecewise linear. This motivates the following definition, given in \cite{JankiewiczNorinWise}.

\begin{definition}
	A square of $\cubulation$ is \newterm{coherent} if parallel edges are oriented in the same direction.
\end{definition}

\newboolean{bad}
\newboolean{terrible}
\begin{figure}
	\centering
	\setboolean{bad}{false}
	\setboolean{terrible}{false}
	\begin{subfigure}{0.3\textwidth}
		\begin{tikzpicture}[scale=2]
	\usetikzlibrary{positioning}
	\usetikzlibrary{calc}
	\usetikzlibrary{backgrounds,fit}
	\usetikzlibrary{patterns}

	\tikzset{
	->-/.style={
	decoration={
			markings,
			mark=at position #1 with {\arrow{Latex[width=2mm,length=3mm]}}
		},
	postaction={decorate}},
	->-/.default=.54
	}

	\useasboundingbox (-1.6,-1.6) rectangle (1.6,1.6);%
	\coordinate (A) at (-1,-1);
	\coordinate (B) at (1,-1);
	\coordinate (C) at (1,1);
	\coordinate (D) at (-1,1);

	\draw
	(0.1,1.3) node{$F_i$}
	(0.1,-1.3) node{$F_j$}
	(-1.3,0.1) node{$F_j$}
	(1.3,0.1) node{$F_i$}

	(-1.3,-1.3) node{$P_v$}
	(1.3,-1.3) node{$P_{v+e_i}$}
	(1.2,1.3) node{$P_{v+e_i+e_j}$}
	(-1.3,1.3) node{$P_{v+e_j}$}

	(0.1,-0.5) node{$I$}
	(-0.1,-0.5) node{$O$}
	(0.5,0.1) node{$O$}
	(0.5,-0.1) node{$I$}
	;

	\draw
	(-1.5,0) -- (1.5,0)
	(0,-1.5) -- (0,1.5);

	\ifbad
		\draw[->-,dashed] (A) -- (B);
		\draw[->-,dashed] (C) -- (D);
		\draw
		(0.1,0.5) node{$O$}
		(-0.1,0.5) node{$I$};
		\ifterrible
			\draw[->-,dashed] (D) -- (A);
			\draw[->-,dashed] (B) -- (C);
			\draw
			(-0.5,0.1) node{$O$}
			(-0.5,-0.1) node{$I$};
		\else
			\draw[->-,dashed] (A) -- (D);
			\draw[->-,dashed] (C) -- (B);
			\draw
			(-0.5,0.1) node{$I$}
			(-0.5,-0.1) node{$O$};
		\fi
		;
	\else
		\draw[->-,dashed] (A) -- (B);
		\draw[->-,dashed] (D) -- (A);
		\draw[->-,dashed] (C) -- (B);
		\draw[->-,dashed] (D) -- (C);
		\draw
		(0.1,0.5) node{$I$}
		(-0.1,0.5) node{$O$}
		(-0.5,0.1) node{$O$}
		(-0.5,-0.1) node{$I$}
		;
	\fi

\end{tikzpicture}
	\end{subfigure}
	\setboolean{bad}{true}
	\begin{subfigure}{0.3\textwidth}
		\begin{tikzpicture}[scale=2]
	\usetikzlibrary{positioning}
	\usetikzlibrary{calc}
	\usetikzlibrary{backgrounds,fit}
	\usetikzlibrary{patterns}

	\tikzset{
	->-/.style={
	decoration={
			markings,
			mark=at position #1 with {\arrow{Latex[width=2mm,length=3mm]}}
		},
	postaction={decorate}},
	->-/.default=.54
	}

	\useasboundingbox (-1.6,-1.6) rectangle (1.6,1.6);%
	\coordinate (A) at (-1,-1);
	\coordinate (B) at (1,-1);
	\coordinate (C) at (1,1);
	\coordinate (D) at (-1,1);

	\draw
	(0.1,1.3) node{$F_i$}
	(0.1,-1.3) node{$F_j$}
	(-1.3,0.1) node{$F_j$}
	(1.3,0.1) node{$F_i$}

	(-1.3,-1.3) node{$P_v$}
	(1.3,-1.3) node{$P_{v+e_i}$}
	(1.2,1.3) node{$P_{v+e_i+e_j}$}
	(-1.3,1.3) node{$P_{v+e_j}$}

	(0.1,-0.5) node{$I$}
	(-0.1,-0.5) node{$O$}
	(0.5,0.1) node{$O$}
	(0.5,-0.1) node{$I$}
	;

	\draw
	(-1.5,0) -- (1.5,0)
	(0,-1.5) -- (0,1.5);

	\ifbad
		\draw[->-,dashed] (A) -- (B);
		\draw[->-,dashed] (C) -- (D);
		\draw
		(0.1,0.5) node{$O$}
		(-0.1,0.5) node{$I$};
		\ifterrible
			\draw[->-,dashed] (D) -- (A);
			\draw[->-,dashed] (B) -- (C);
			\draw
			(-0.5,0.1) node{$O$}
			(-0.5,-0.1) node{$I$};
		\else
			\draw[->-,dashed] (A) -- (D);
			\draw[->-,dashed] (C) -- (B);
			\draw
			(-0.5,0.1) node{$I$}
			(-0.5,-0.1) node{$O$};
		\fi
		;
	\else
		\draw[->-,dashed] (A) -- (B);
		\draw[->-,dashed] (D) -- (A);
		\draw[->-,dashed] (C) -- (B);
		\draw[->-,dashed] (D) -- (C);
		\draw
		(0.1,0.5) node{$I$}
		(-0.1,0.5) node{$O$}
		(-0.5,0.1) node{$O$}
		(-0.5,-0.1) node{$I$}
		;
	\fi

\end{tikzpicture}
	\end{subfigure}
	\setboolean{terrible}{true}
	\begin{subfigure}{0.3\textwidth}
		\begin{tikzpicture}[scale=2]
	\usetikzlibrary{positioning}
	\usetikzlibrary{calc}
	\usetikzlibrary{backgrounds,fit}
	\usetikzlibrary{patterns}

	\tikzset{
	->-/.style={
	decoration={
			markings,
			mark=at position #1 with {\arrow{Latex[width=2mm,length=3mm]}}
		},
	postaction={decorate}},
	->-/.default=.54
	}

	\useasboundingbox (-1.6,-1.6) rectangle (1.6,1.6);%
	\coordinate (A) at (-1,-1);
	\coordinate (B) at (1,-1);
	\coordinate (C) at (1,1);
	\coordinate (D) at (-1,1);

	\draw
	(0.1,1.3) node{$F_i$}
	(0.1,-1.3) node{$F_j$}
	(-1.3,0.1) node{$F_j$}
	(1.3,0.1) node{$F_i$}

	(-1.3,-1.3) node{$P_v$}
	(1.3,-1.3) node{$P_{v+e_i}$}
	(1.2,1.3) node{$P_{v+e_i+e_j}$}
	(-1.3,1.3) node{$P_{v+e_j}$}

	(0.1,-0.5) node{$I$}
	(-0.1,-0.5) node{$O$}
	(0.5,0.1) node{$O$}
	(0.5,-0.1) node{$I$}
	;

	\draw
	(-1.5,0) -- (1.5,0)
	(0,-1.5) -- (0,1.5);

	\ifbad
		\draw[->-,dashed] (A) -- (B);
		\draw[->-,dashed] (C) -- (D);
		\draw
		(0.1,0.5) node{$O$}
		(-0.1,0.5) node{$I$};
		\ifterrible
			\draw[->-,dashed] (D) -- (A);
			\draw[->-,dashed] (B) -- (C);
			\draw
			(-0.5,0.1) node{$O$}
			(-0.5,-0.1) node{$I$};
		\else
			\draw[->-,dashed] (A) -- (D);
			\draw[->-,dashed] (C) -- (B);
			\draw
			(-0.5,0.1) node{$I$}
			(-0.5,-0.1) node{$O$};
		\fi
		;
	\else
		\draw[->-,dashed] (A) -- (B);
		\draw[->-,dashed] (D) -- (A);
		\draw[->-,dashed] (C) -- (B);
		\draw[->-,dashed] (D) -- (C);
		\draw
		(0.1,0.5) node{$I$}
		(-0.1,0.5) node{$O$}
		(-0.5,0.1) node{$O$}
		(-0.5,-0.1) node{$I$}
		;
	\fi

\end{tikzpicture}
	\end{subfigure}
	\caption{The two possible orientations on a square producing a $1$-cocycle, up to rotation.} %
	\label{fig:square-orientations}
\end{figure}

If the set of moves is sparse, then every square is coherent, and we can extend $\foneskel$ on every cube via a diagonal map, as done in \cite{JankiewiczNorinWise}. Here, however, the situation is more complicated.

Let $\pol[v]$ be a copy of $\pol$, and let $\facet[i]$ and $\facet[j]$ be two adjacent facets. Their intersection is dual to a square, as in \cref{fig:square-orientations}. We have three cases:

\begin{itemize}
	\item If they belong to different moves, then the square is coherent, as in \cref{fig:square-orientations}, left.
	\item If they belong to the same move, and they have the same status, then the square is as in \cref{fig:square-orientations}, centre (up to a rotation).
	\item If they belong to the same move, and they have different status, then we get a square as in \cref{fig:square-orientations}, right.
\end{itemize}

Note that if we have squares of the third kind, then the $1$-cochain is not a $1$-cocycle, so there is no hope to extend $\foneskel$ to the whole $\cubulation$. The following definition ensures that we avoid this case.

\begin{definition}
	A state $s$ is \newterm{compatible} with a set of moves if whenever $ F_i $ and $F_j$ are adjacent and belong to the same move, then $ s(F_i)=s(F_j) $. %
\end{definition}
\begin{remark}
	If a state $s$ is compatible, than the whole orbit under the actions of $\vectorspace$ via the moves is made of compatible states. This ensures that it is sufficient to check compatibility of the initial state.
\end{remark}

From the analysis above, we get the following simple consequence.

\begin{lemma}
	The map $ \foneskel \colon \skel1 \cubulation \to S^1 $ extends continuously to $\cubulation$ if and only if the initial state is compatible.
\end{lemma}

However, to extend it to a Bestvina-Brady Morse function, we need a subdivision of the cubes. This had already been done in \cite{IMMHyperbolic5Manifold}, by subdividing the squares that are not coherent into four triangles: since here the situation is more complicated, we will work more abstractly and perform a full barycentric subdivision. We will explain in \cref{sec:bary-sub} how to define a Bestvina-Brady Morse function on this subdivision.

\subsection{Bad faces and state inheritance}\label{sub:state-inheritance}

After performing the subdivision, a new vertex appears at the barycentre of every cell of the tessellation. To study the ascending and descending links at these new vertices, we need to introduce two definitions of combinatorial nature on the faces of $ \pol $: the notion of \emph{good} and \emph{bad} face, and the \emph{inherited state}.

\begin{definition}
	Let $\face$ be a face of $\pol$ of codimension $ \ell $. We call the facets of $ \pol $ that contain $ \face $ the \newterm{defining facets of $ \face $}.
\end{definition}

Note that $\face$ is the intersection of all its defining facets. Let $ \face $ be a face of $P$, with defining facets $ \facet[i_1], \dots, \facet[i_\ell] $.

\begin{definition}\label{def:good-polytope}
	We call the face $\face$ \newterm{good} if there is a move that contains exactly one of its defining facets. We call it \newterm{bad} otherwise.
\end{definition}

A state on the polytope $\pol$ induces states on all its faces, as follows.

\begin{definition}\label{def:inherited-state}
	Let $ \pol $ be a polytope equipped with a state $\state$, and $\face$ be a face. For a facet $ \varfacet $ of $ \face $, consider the facet $ \facet $ of $P$ such that $ \varfacet = \facet \cap \face $.

	The \newterm{inherited state} on $ \face $ assigns to $ \varfacet $ a status as follows:

	\begin{itemize}
		\item it assigns $\status\facet $ if $\facet$ is in a different move than all the defining facets of $ \face $;
		\item it assigns status $\Out$ otherwise, independently of the status of $ \facet $.
	\end{itemize}
\end{definition}

Even though this definition could seem arbitrary, the reason for which we need to sometimes assign status $\Out$ will become apparent in \cref{sec:P6}.

\begin{remark}\label{rem:p-bad}
	If $\face = \pol$, then the inherited state coincides with the usual state. Moreover, $ \pol $ is a bad face of itself (since the set of defining facets is empty).
\end{remark}

For now, we considered $\face$ to be a face of an abstract copy of the polytope $\pol$; there are however several copies of $\face$ inside the tessellation of $\mfld$. Given a certain copy of $\face$, we can choose some copy $\pol[v]$ of $\pol$ that contains it. In particular, this copy of $\face$ is the intersection of all the copies of $P$ of the form $\pol[v+w]$, where $ w \in \generatedby{e_{i_1}, \dots, e_{i_\ell}} $. There is therefore a natural correspondence between copies of $\face$ and vectors of $ \quotient{\vectorspace*}{\generatedby{e_{i_1}, \dots, e_{i_\ell}}} $, so we will denote a copy of $\face$ by $ \face[v] $, where $[v]$ is the projection to the quotient of $ v \in \vectorspace $.

We assign to $ \face[v] $ the inherited state from $ \pol[v] $.

\begin{lemma}
	The inherited state $ \inheritedstate[v] $ is well-defined, that is, if $ v,w $ are two vectors that project to the same element $ [v]=[w] $ in the quotient, then $\inheritedstate[v]=\inheritedstate[w]$.
\end{lemma}

\begin{proof}
	Consider a facet $ \varfacet = \face[v] \cap \facet $ of $\pol[v]$. If $\facet$ belongs to a move that contains one of the $\facet[i_j]$, then $ \inheritedstatus[v]\varfacet = \inheritedstatus[v]\varfacet = \Out $. This is independent of $v$, since it is defined using only data from the abstract polytope $P$.

	Otherwise, suppose that $\facet$ does not belong to the same move as any of the defining facets $\facet[i_1], \dots, \facet[i_\ell]$. In this case, the action of the defining facets does not change the status of $F$. Since $w$ can be obtained from $v$ by adding some $ e_{i_j} $, and doing so does not change the status of $\facet$ by what we just assumed, we get that $ \status[v]\facet= \status[w]\facet $, and therefore by definition $ \inheritedstatus[v]\varfacet = \status[v]\facet = \status[w]\facet = \inheritedstatus[w]\varfacet $.
\end{proof}

Note that the lemma above would not hold if we had defined the inherited state just by looking at the status of the corresponding facets of $\pol$, without taking the moves into account.

\section{The barycentric subdivision}\label{sec:bary-sub}

Let $\pol$ be a hyperbolic right-angled polytope, equipped with a set of moves and a compatible initial state. By using the Löbell construction described in the previous section, we obtain a manifold $ \mfld $ tessellated into copies $ \pol[v] $ of $ \pol $, where $ v \in \vectorspace $. Dual to the tessellation of $ \mfld $, we have a cube complex $ \cubulation $.

Given any cell complex $X$, denote by $ \sd X $ the barycentric subdivision of $X$.

The aim of this section is to prove the following theorem.

\begin{theorem} \label{thm:inherited-state-collapsibility}
	Let $\pol$ and $ \cubulation $ as above. There exists a Bestvina-Brady Morse function $ f \colon \sd \cubulation \to S^1  $
	such that the following holds for any copy $\face[v]$ of some face $\face$ of $\pol$.

	\begin{itemize}
		\item If the inherited state on $ \face[v] $ is $ m $-legal for some $ m \in \NN $, then the ascending and descending link of $f$ at the barycentre of $\face$ are $m$-connected.
		\item If either $ \face  $ is good \emph{or} the inherited state is $ \infty $-legal, the ascending and descending links are collapsible.
	\end{itemize}
\end{theorem}

\begin{remark}
	Note that if we only consider sparse set of moves, we get the original theory of Jankiewicz--Norin--Wise: all the proper faces of $P$ are good, so the vertices of $ \sd C $ that were introduced by the barycentric subdivision have collapsible ascending and descending link. A vertex $v$ of the original cubulation corresponds dually to a copy $\pol[v]$ of $\pol$, so the ascending and descending links are determined by the state on $\pol[v]$.
\end{remark}

In the previous section, we used the states to build a Bestvina-Brady Morse function $ \foneskel \colon \skel1\cubulation \to S^1 $ defined on the $1$-skeleton of $ \cubulation $.
Starting from $\foneskel$, we define a Bestvina-Brady Morse function $ \baryf \colon \sd\cubulation \to S^1 $ on the barycentric subdivision of $ \cubulation $.

Fix a small $ \epsilon > 0 $. We start by letting $ \baryf(v) = 0 $ for every vertex $ v \in \sd\cubulation $ that comes from the original cubulation. Then we identify each oriented edge with $ [0,1] $, where the barycentre corresponds to $ \frac12 $, and we define $f$ on that edge via the map $ [0,1] \mapsto [0,1]/\sim $ that sends $ [0,\frac12] $ to $ [0,\epsilon] $ and $ [\frac12, 1] $ to $ [\epsilon,1] $ linearly.

Suppose that we have already defined $\baryf$ on the barycentric subdivision of the $(k-1)$-skeleton $ \skel{k-1}{\cubulation} $ of $\cubulation$. Let $\cube$ be a $k$-cube of $\cubulation$, and let $ \barycentre\cube $ denote the barycentre of $\cube$. We may lift $ f \colon \sd(\skel{k-1}\cube) \to S^1 $ to a map $ \lift f \colon \sd(\skel{k-1} \cube) \to \RR $; here we are using that $\foneskel$ comes from a cocycle when $ k=2 $.
We define
\[
	\lift\baryf(\barycentre\cube) = \min \set*{\restrict {\lift f}{\boundary \cube}} + k \epsilon,
\]
and we extend it linearly to all the simplices of $ \sd \cubulation $. Then we pass it to the quotient $ \baryf \colon \cube \to S^1 $: it is easy to see that the resulting $\baryf$ does not depend on the choice of the lift. A concrete example of the construction can be seen in \cref{fig:low-barycentric}.

\begin{remark}
	Note that, in the previous definition, the minimum is always attained at a vertex of $\cube$.
\end{remark}

\newboolean{subdivide}
\newboolean{showlink}
\begin{figure}
	\centering
	\setboolean{subdivide}{false}
	\begin{subfigure}{0.4\textwidth}
		\begin{tikzpicture}[scale=3]
	\usetikzlibrary{positioning}
	\usetikzlibrary{calc}
	\usetikzlibrary{backgrounds,fit}
	\usetikzlibrary{patterns}

	\tikzset{
	-|>-/.style={
	decoration={
			markings,
			mark=at position #1 with {\arrow{latex}}
		},
	postaction={decorate}},
	-|>-/.default=.55
	}

	\tikzset{
	->-/.style={
	decoration={
			markings,
			mark=at position #1 with {\arrow{>}}
		},
	postaction={decorate}},
	->-/.default=.55
	}
	\tikzset{
	->>-/.style={
	decoration={
			markings,
			mark=at position #1 with {\arrow{>>}}
		},
	postaction={decorate}},
	->>-/.default=.53
	}
	\tikzset{
	->>>>-/.style={
	decoration={
			markings,
			mark=at position #1 with {\arrow{>>>>}}
		},
	postaction={decorate}},
	->>>>-/.default=.58
	}

	\useasboundingbox (-.4,-.3) rectangle (1.3,1.3);%
	\coordinate (A) at (0,0);
	\coordinate (B) at (1,0);
	\coordinate (C) at (1,1);
	\coordinate (D) at (0,1);
	\coordinate (E) at (0.5, 0);
	\coordinate (F) at (1, 0.5);
	\coordinate (G) at (0.5, 1);
	\coordinate (H) at (0, 0.5);
	\coordinate (I) at (0.5, 0.5);

	\ifshowlink
		\draw[red,fill] (0,0) circle (0.03);
	\fi

	\ifsubdivide
		\tiny
		\node[below left] at (A) {$0$};
		\node[below right] at (B) {$1$};
		\node[above right] at (C) {$2$};
		\node[above left] at (D) {$1$};
		\node[below] at (E) {$\epsilon$};
		\node[right] at (F) {$1+\epsilon$};
		\node[above] at (G) {$1+\epsilon$};
		\node[left] at (H) {$\epsilon$};

		\node[above right=-0.05 and .1] at (I) {$2\epsilon$};

		\draw[->-] (A) -- (E);
		\draw[->-] (E) -- (B);
		\draw[->-] (A) -- (H);
		\draw[->-] (H) -- (D);

		\draw[->-] (B) -- (F);
		\draw[->-] (F) -- (C);
		\draw[->-] (D) -- (G);
		\draw[->-] (G) -- (C);

		\draw[->-] (E) -- (I);
		\draw[->-] (H) -- (I);
		\draw[->-] (I) -- (F);
		\draw[->-] (I) -- (G);

		\ifshowlink
			\draw[->-, red] (A) -- (I);
		\else
			\draw[->-] (A) -- (I);
		\fi

		\draw[->-] (I) -- (B);
		\draw[->-] (I) -- (C);
		\draw[->-] (I) -- (D);

	\else
		\scriptsize
		\node[below left=.1cm] at (A) {$0$};
		\node[below right=.1cm] at (B) {$1$};
		\node[above right=.1cm] at (C) {$2$};
		\node[above left=.1cm] at (D) {$1$};
		\draw[->-] (A) -- (B);
		\draw[->-] (A) -- (D);
		\draw[->-] (B) -- (C);
		\draw[->-] (D) -- (C);
	\fi

\end{tikzpicture}
	\end{subfigure}
	\setboolean{subdivide}{true}
	\begin{subfigure}{0.4\textwidth}
		\begin{tikzpicture}[scale=3]
	\usetikzlibrary{positioning}
	\usetikzlibrary{calc}
	\usetikzlibrary{backgrounds,fit}
	\usetikzlibrary{patterns}

	\tikzset{
	-|>-/.style={
	decoration={
			markings,
			mark=at position #1 with {\arrow{latex}}
		},
	postaction={decorate}},
	-|>-/.default=.55
	}

	\tikzset{
	->-/.style={
	decoration={
			markings,
			mark=at position #1 with {\arrow{>}}
		},
	postaction={decorate}},
	->-/.default=.55
	}
	\tikzset{
	->>-/.style={
	decoration={
			markings,
			mark=at position #1 with {\arrow{>>}}
		},
	postaction={decorate}},
	->>-/.default=.53
	}
	\tikzset{
	->>>>-/.style={
	decoration={
			markings,
			mark=at position #1 with {\arrow{>>>>}}
		},
	postaction={decorate}},
	->>>>-/.default=.58
	}

	\useasboundingbox (-.4,-.3) rectangle (1.3,1.3);%
	\coordinate (A) at (0,0);
	\coordinate (B) at (1,0);
	\coordinate (C) at (1,1);
	\coordinate (D) at (0,1);
	\coordinate (E) at (0.5, 0);
	\coordinate (F) at (1, 0.5);
	\coordinate (G) at (0.5, 1);
	\coordinate (H) at (0, 0.5);
	\coordinate (I) at (0.5, 0.5);

	\ifshowlink
		\draw[red,fill] (0,0) circle (0.03);
	\fi

	\ifsubdivide
		\tiny
		\node[below left] at (A) {$0$};
		\node[below right] at (B) {$1$};
		\node[above right] at (C) {$2$};
		\node[above left] at (D) {$1$};
		\node[below] at (E) {$\epsilon$};
		\node[right] at (F) {$1+\epsilon$};
		\node[above] at (G) {$1+\epsilon$};
		\node[left] at (H) {$\epsilon$};

		\node[above right=-0.05 and .1] at (I) {$2\epsilon$};

		\draw[->-] (A) -- (E);
		\draw[->-] (E) -- (B);
		\draw[->-] (A) -- (H);
		\draw[->-] (H) -- (D);

		\draw[->-] (B) -- (F);
		\draw[->-] (F) -- (C);
		\draw[->-] (D) -- (G);
		\draw[->-] (G) -- (C);

		\draw[->-] (E) -- (I);
		\draw[->-] (H) -- (I);
		\draw[->-] (I) -- (F);
		\draw[->-] (I) -- (G);

		\ifshowlink
			\draw[->-, red] (A) -- (I);
		\else
			\draw[->-] (A) -- (I);
		\fi

		\draw[->-] (I) -- (B);
		\draw[->-] (I) -- (C);
		\draw[->-] (I) -- (D);

	\else
		\scriptsize
		\node[below left=.1cm] at (A) {$0$};
		\node[below right=.1cm] at (B) {$1$};
		\node[above right=.1cm] at (C) {$2$};
		\node[above left=.1cm] at (D) {$1$};
		\draw[->-] (A) -- (B);
		\draw[->-] (A) -- (D);
		\draw[->-] (B) -- (C);
		\draw[->-] (D) -- (C);
	\fi

\end{tikzpicture}
	\end{subfigure}
	\caption{On the left, a square $Q$ with its edges oriented; the number on the vertices are the values of the lift of $ \foneskel $ to $ \RR $. On the right, the values of the lift $ \lift\baryf \colon \sd Q \to \RR $.}
	\label{fig:low-barycentric}
\end{figure}

This procedure defines inductively a Bestvina-Brady map $ \baryf \colon \sd\cubulation \to S^1 $.
To show that $ \baryf $ satisfies the hypotheses of \cref{thm:inherited-state-collapsibility}, we employ a well-known decomposition of links in barycentric subdivision into a join of two simplicial complexes, as follows (compare for example \cite{DavisOkun}*{Lemma 2.1}).

Let $\cube$ be a cube of $\cubulation$. The link of its barycentre $\barycentre\cube$ is the full subcomplex of $ \sd\cubulation $ whose vertices are all the barycentres of the proper faces and cofaces of $\cube$.

This motivates the following definition.

\begin{definition}
	We call \newterm{face link of $\cube$}, and denote with $ \flink\cube $, the full subcomplex of $ \sd\cubulation $ whose vertices are barycentres of proper faces of $\cube$. Similarly, we call \newterm{coface link of $\cube$} the full subcomplex of $ \sd\cubulation $ whose vertices are barycentres of proper cofaces of $\cube$; we denote it with $ \coflink\cube $.
\end{definition}

\begin{remark}
	The face link $\flink\cube$ coincides with the barycentric subdivision of $ \boundary\cube $. \end{remark}

\begin{lemma}
	The link $ \link(\barycentre\cube, \sd\cubulation) $ is the join $ \flink\cube * \coflink\cube $.
\end{lemma}

\begin{proof}
	Follows directly from the definition of barycentric subdivision.
\end{proof}

We can obtain a similar statement for the ascending and descending links by intersecting them with the decomposition given above.

\begin{definition}
	The \newterm{ascending} (resp.~\newterm{descending}) \newterm{face link} of $\cube$ is the intersection of $\flink\cube$ with the ascending (resp.~descending) link of $ \baryf $ at $ \barycentre \cube $.s The \newterm{ascending} (resp.~\newterm{descending}) \newterm{coface link} of $\cube$ is defined analogously by intersecting $ \coflink\cube $ with the ascending (resp.~descending) link at $ \barycentre \cube $.

	We denote the ascending and descending face links with $ \flink_{\uparrow/\downarrow}\cube $, and the ascending/descending coface links with $ \coflink_{\uparrow/\downarrow}\cube $.
\end{definition}

\begin{lemma}\label{lemma:linkup-join-flinkup-coflinkup}
	The ascending link $ \linkup[\baryf]{\barycentre\cube} $ is the join $ \flinkup\cube * \coflinkup \cube $. Analogously, the descending link $ \linkdown{\barycentre\cube} $ is the join $ \flinkdown\cube * \coflinkdown\cube $.
\end{lemma}

This way, to show that the ascending link of $\barycentre\cube$ is collapsible, it suffices to show that at least one of $ \flinkup\cube $ and $ \coflinkup\cube $ is collapsible. We give an algorithmic way to compute both: we will show that the goodness of a face controls the face link, while the inherited state controls the coface link.

\subsection{Ascending and descending face links}

Let $\cube$ be a $k$-cube of $\cubulation$. This is dual to some cell of the tessellation of $\mfld$ into copies of $\pol$; in particular, it is a face $ \face[v] $ of some $ \pol[v] $, where $v$ is a vector of $\vectorspace$. We denote by $ \deffacets $ the set of defining facets for $\face$, so $ \face=\bigcap \deffacets $.

Our aim is to compute ascending and descending face links of $\cube$. To do this, we need to study the restriction $ \restrict \baryf\cube \colon \cube \to S^1 $; since $\cube$ is simply connected, it is more convenient to study the lift $ \lift \baryf \colon \cube \to \RR $. We conventionally choose the lift such that it has minimum at $0$.

Let us start by analyzing the simple case where all the facets of $\deffacets$ belong to the same move. We call such a cube \emph{monochromatic}. By the compatibility of the state, all the edges coming out of a vertex of $\cube$ are all either oriented inward or all oriented outwards, so we get the pictures in \cref{fig:bad-cube}.
\begin{remark}
	All edges of $\cubulation$ are by definition monochromatic $1$-cubes.
\end{remark}

\newboolean{threedim}
\begin{figure}
	\centering
	\begin{subfigure}{0.4\textwidth}
		\begin{tikzpicture}[scale=3]
	\usetikzlibrary{positioning}
	\usetikzlibrary{calc}
	\usetikzlibrary{backgrounds,fit}
	\usetikzlibrary{patterns}

	\tikzset{
	-|>-/.style={
	decoration={
			markings,
			mark=at position #1 with {\arrow{latex}}
		},
	postaction={decorate}},
	-|>-/.default=.5
	}

	\tikzset{
	->-/.style={
	decoration={
			markings,
			mark=at position #1 with {\arrow{>}}
		},
	postaction={decorate}},
	->-/.default=.53
	}

	\ifthreedim
		\useasboundingbox (-0.5,-0.5) rectangle (1.3,1.3);%
	\else
		\useasboundingbox (-0.2,-0.2) rectangle (1.3,1.3);%
	\fi
	\coordinate (A) at (0,0,0);
	\coordinate (B) at (1,0,0);
	\coordinate (C) at (1,1,0);
	\coordinate (D) at (0,1,0);
	\coordinate (E) at (0,0,1);
	\coordinate (F) at (1,0,1);
	\coordinate (G) at (1,1,1);
	\coordinate (H) at (0,1,1);

	\node[above left] at (A) {$0$};
	\node[above right] at (B) {$1$};
	\node[above right] at (C) {$0$};
	\node[above left] at (D) {$1$};
	\ifthreedim
		\node[below left] at (E) {$1$};
		\node[below right] at (F) {$0$};
		\node[above left] at (G) {$1$};
		\node[above left] at (H) {$0$};
	\fi

	\draw[->-] (A) -- (B);
	\draw[->-] (A) -- (D);
	\draw[->-] (C) -- (B);
	\draw[->-] (C) -- (D);

	\ifthreedim
		\draw[->-] (A) -- (E);
		\draw[->-] (F) -- (B);
		\draw[->-] (C) -- (G);
		\draw[->-] (H) -- (D);

		\draw[->-] (F) -- (E);
		\draw[->-] (H) -- (E);
		\draw[->-] (F) -- (G);
		\draw[->-] (H) -- (G);
	\fi

\end{tikzpicture}
	\end{subfigure}
	\setboolean{threedim}{true}
	\begin{subfigure}{0.4\textwidth}
		\begin{tikzpicture}[scale=3]
	\usetikzlibrary{positioning}
	\usetikzlibrary{calc}
	\usetikzlibrary{backgrounds,fit}
	\usetikzlibrary{patterns}

	\tikzset{
	-|>-/.style={
	decoration={
			markings,
			mark=at position #1 with {\arrow{latex}}
		},
	postaction={decorate}},
	-|>-/.default=.5
	}

	\tikzset{
	->-/.style={
	decoration={
			markings,
			mark=at position #1 with {\arrow{>}}
		},
	postaction={decorate}},
	->-/.default=.53
	}

	\ifthreedim
		\useasboundingbox (-0.5,-0.5) rectangle (1.3,1.3);%
	\else
		\useasboundingbox (-0.2,-0.2) rectangle (1.3,1.3);%
	\fi
	\coordinate (A) at (0,0,0);
	\coordinate (B) at (1,0,0);
	\coordinate (C) at (1,1,0);
	\coordinate (D) at (0,1,0);
	\coordinate (E) at (0,0,1);
	\coordinate (F) at (1,0,1);
	\coordinate (G) at (1,1,1);
	\coordinate (H) at (0,1,1);

	\node[above left] at (A) {$0$};
	\node[above right] at (B) {$1$};
	\node[above right] at (C) {$0$};
	\node[above left] at (D) {$1$};
	\ifthreedim
		\node[below left] at (E) {$1$};
		\node[below right] at (F) {$0$};
		\node[above left] at (G) {$1$};
		\node[above left] at (H) {$0$};
	\fi

	\draw[->-] (A) -- (B);
	\draw[->-] (A) -- (D);
	\draw[->-] (C) -- (B);
	\draw[->-] (C) -- (D);

	\ifthreedim
		\draw[->-] (A) -- (E);
		\draw[->-] (F) -- (B);
		\draw[->-] (C) -- (G);
		\draw[->-] (H) -- (D);

		\draw[->-] (F) -- (E);
		\draw[->-] (H) -- (E);
		\draw[->-] (F) -- (G);
		\draw[->-] (H) -- (G);
	\fi

\end{tikzpicture}
	\end{subfigure}
	\caption{On the left, a monochromatic square; on the right, a monochromatic cube. The numbers on the vertices denote the values of $ \lift f $ on the vertices.}
	\label{fig:bad-cube}
\end{figure}

We may explicitly compute the value of $ \lift \baryf $ at the vertices of the barycentric subdivision of a monochromatic cube $\cube$. On the vertices of the original cube, the function assumes values $0$ and $1$ in a checkerboard-like pattern. Every positive-dimensional face $\varcube$ of $\cube$ contains a vertex $v$ such that $ \lift\baryf(v) = 0 $, so $ \lift f $ assumes the value $ h \epsilon $ at the barycentres of all $h$-faces of $Q$, for all $1 \leq h \leq k$.

We get the following lemma.

\begin{lemma}\label{lemma:face-link-monochromatic}
	Let $\cube$ be a monochromatic $k$-dimensional cube. Then the ascending face link consists of $ 2^{k-1} $ disjoint points, while the descending face link is obtained from $ \sd \boundary\cube $ by removing the open stars of $ 2^{k-1} $ vertices of $\cube$.

	In particular, ascending and descending face links of $1$-cubes are collapsible, while ascending and descending face links of a monochromatic square collapse to $S^0$ (i.e.~a disjoint union of two points).
\end{lemma}

\begin{proof}
	Since $ \lift\baryf(\barycentre\cube) = k\epsilon $, the only barycentres of faces on which $ \lift\baryf $ assumes a higher value are the vertices of the cube on which $ \baryf $ assumes value $1$. So the ascending face link contains exactly half of the vertices of $\cube$. The descending face link is the full subcomplex of $ \sd \boundary \cube $ with vertex set consisting of all vertices that do not belong to the ascending face link: so it is obtained by removing from $ \sd \boundary Q $ the open stars of vertices in the ascending link.
\end{proof}

We now turn our attention to the general case. Given a cube $\cube$, we can always decompose it as a product $ \cube[1] \times \dots \times \cube[\ell] $, where every factor is monochromatic. This decomposition is obtained by looking at how the moves partition the set $\deffacets$; in particular, a $2$-dimensional face of $\cube$ is coherent if and only it is not contained in one of the factors.

Consider a product $ \cube = \cube[1] \times \cube[2] $ of two monochromatic cubes. If $ v, v' $ are vertices of $ \cube[2] $, the orientation of the edges on $ \cube[1] \times \set v $ and $ \cube[1] \times \set{v'} $ coincide, so $ \baryf(x,v) = \baryf(x,v') $ for all $ x \in \cube[1] $. For this reason, whenever we have a product $ \cube = \cube[1] \times \dots \times \cube[\ell] $  of monochromatic cubes, we have a well-defined restriction $ \restrict \baryf{\cube[i]} $, obtained by restricting $ \baryf $ to any face of the form $ \set{v_1} \times \dots \times \cube[i] \times \dots \times \set{v_\ell}  $, where $ v_j $ is a vertex of $ \cube[j] $ for all $ j \neq i $.

\begin{lemma}\label{sum-decomposition}
	Let $ \cube $ be a cube of $ \cubulation $ that decomposes as a product of monochromatic cubes $ \cube[1] \times \dots \times \cube[\ell] $. Then
	\[
		\baryf(x_1, \ldots, x_\ell) = \sum_{i=1}^\ell \restrict \baryf{\cube[i]}(x_i)
	\]
	for every $ (x_1, \ldots, x_\ell) \in \cube[1] \times \dots \times \cube[\ell] $.
\end{lemma}

\begin{proof}
	Lift every $ \restrict{\baryf}{\cube[i]} $ to a map $ \lift f_i \colon \cube[i] \to \RR $, in such a way that their minimum is $0$. Lift also $ \restrict \baryf\cube $ to $ \lift \baryf \colon \cube \to \RR $, with the same condition on the minimum. We claim that $ \lift \baryf(x_1, \dots, x_\ell) = \sum_{i=1}^{\ell} \lift f_i(x_i) $; once proved this, the lemma follows by passing to the quotient.

	Let $ x_1, \dots, x_\ell $ be vertices of $ \sd \cube[1], \dots, \sd \cube[\ell] $. Choose an $ i \in \range \ell $, and let $ x'_i $ be another vertex of $ \sd\cube[i] $.

	Then we have that
	\[
		\lift \baryf(x_1, \dots, x_i, \dots, x_\ell) -  \lift \baryf(x_1, \dots, x'_i, \dots, x_\ell) = \lift f_i(x_i) - \lift f_i(x'_i).
	\]
	This implies that the two members of the claimed equality coincide up to an integer constant. By the condition on the minimum, the constant must be $0$, so we conclude that the claim holds whenever the $x_i$ are vertices of the barycentric subdivision.

	Since we are taking the unique affine extension on each simplex of $ \sd\cube $, the claim holds for every point in $\cube$.
\end{proof}

We now translate \cref{def:good-polytope} to the dual cube complex.

\begin{definition}
	We call a $k$-cube $ \cube $ \newterm{good} if it is dual to a good face of $P$ (as per \cref{def:good-polytope}), and \newterm{bad} otherwise.
\end{definition}

\begin{remark}
	As noted in \cref{rem:p-bad}, the polytope $\pol$ is always bad; therefore all vertices of $\cubulation$ are bad, since that are dual to some copy of $\pol$.
\end{remark}

A $k$-cube is therefore good if and only if there is a move containing exactly one facet of $ \deffacets $; in other words, when at least one of its monochromatic factors is $1$-dimensional.

In \cref{fig:bad-square-times-interval} we have an example of a product of a bad square with an interval, which is a good cube.

\newboolean{subdivision}
\begin{figure}
	\centering
	\begin{subfigure}{0.4\textwidth}
		\begin{tikzpicture}[scale=3]

	\tikzset{
	-|>-/.style={
	decoration={
			markings,
			mark=at position #1 with {\arrow{latex}}
		},
	postaction={decorate}},
	-|>-/.default=.5
	}

	\tikzset{
	->-/.style={
	decoration={
			markings,
			mark=at position #1 with {\arrow{>}}
		},
	postaction={decorate}},
	->-/.default=.53
	}

	\coordinate (A) at (0,0,0);
	\coordinate (B) at (1,0,0);
	\coordinate (C) at (1,0,1);
	\coordinate (D) at (0,0,1);
	\coordinate (E) at (0,1,0);
	\coordinate (F) at (1,1,0);
	\coordinate (G) at (1,1,1);
	\coordinate (H) at (0,1,1);

	\ifsubdivision\else
	\fi

	\ifsubdivision\else
		\node[above left] at (A) {$0$};
		\node[above right] at (B) {$1$};
		\node[below right] at (C) {$0$};
		\node[below left] at (D) {$1$};
		\node[above left] at (E) {$1$};
		\node[above right] at (F) {$2$};
		\node[above left] at (G) {$1$};
		\node[above left] at (H) {$2$};
	\fi

	\draw[->-] (A) -- (B);
	\draw[->-] (A) -- (D);
	\draw[->-] (C) -- (B);
	\draw[->-] (C) -- (D);

	\draw[->-] (A) -- (E);
	\draw[->-] (B) -- (F);
	\draw[->-] (C) -- (G);
	\draw[->-] (D) -- (H);

	\draw[->-] (E) -- (F);
	\draw[->-] (E) -- (H);
	\draw[->-] (G) -- (F);
	\draw[->-] (G) -- (H);

	\ifsubdivision
		\foreach\x in {0,1} {
				\foreach\y in {0,1} {
						\foreach\z in {0,1} {
								\draw[dotted]
								(\x,\y,\z) -- (\x, \y, 0.5) -- (\x, 0.5, 0.5) -- cycle
								(\x,\y,\z) -- (\x, 0.5, \z) -- (\x, 0.5, 0.5) -- cycle
								(\x,\y,\z) -- (\x, 0.5, \z) -- (0.5, 0.5, \z) -- cycle
								(\x,\y,\z) -- (0.5, \y, \z) -- (0.5, 0.5, \z) -- cycle
								(\x,\y,\z) -- (\x, \y, 0.5) -- (0.5, \y, 0.5) -- cycle
								(\x,\y,\z) -- (0.5, \y, \z) -- (0.5, \y, 0.5) -- cycle
								;
							}
					}
			}
		\draw[red,thick]
		(0,0.5,0.5) -- (0,0.5,0) -- (0.5,0.5,0)
		(1,0.5,0.5) -- (1,0.5,1) -- (0.5,0.5,1);

		\draw[red, dashed]
		(0,0.5,0.5) -- (0,0.5,0) -- (0.5,0.5,0) -- (0.5,0,0) -- (0.5,0,0.5) -- (0,0,0.5) -- cycle
		(1,0.5,0.5) -- (1,0.5,1) -- (0.5,0.5,1) -- (0.5,0,1) -- (0.5,0,0.5) -- (1,0,0.5) -- cycle
		;

		\draw[red,fill] (0.5,0,0.5) circle (0.01);

		\path[pattern=north west lines, pattern color=red] (0,0,0.5) -- (0,0.5,0.5) -- (0,0.5,0) -- (0,0,0) -- cycle;
		\path[pattern=north west lines, pattern color=red] (0,0.5,0) -- (0.5,0.5,0) -- (0.5,0,0) -- (0,0,0) -- cycle;
		\path[pattern=north west lines, pattern color=red] (0.5,0,0) -- (0.5,0,0.5) -- (0,0,0.5) -- (0,0,0) -- cycle;

		\path[pattern=north west lines, pattern color=red] (1,0,0.5) -- (1,0.5,0.5) -- (1,0.5,1) -- (1,0,1) -- cycle;
		\path[pattern=north west lines, pattern color=red] (1,0.5,1) -- (0.5,0.5,1) -- (0.5,0,1) -- (1,0,1) -- cycle;
		\path[pattern=north east lines, pattern color=red] (0.5,0,1) -- (0.5,0,0.5) -- (1,0,0.5) -- (1,0,1) -- cycle;

		\draw[blue]
		(0,0,1) -- (0,1,1)
		(1,0,0) -- (1,1,0)
		;

		\draw[blue, pattern=north west lines, pattern color=blue]
		(1,1,0) -- (1,1,1) -- (0,1,1) -- (0,1,0) -- cycle;
		\draw[blue,fill] (0.5,1,0.5) circle (0.01);
		\draw[blue, thick] (0,0.5,1) -- (0,1,1) -- (1,1,0) -- (1,0.5,0);
	\fi

\end{tikzpicture}
	\end{subfigure}
	\setboolean{subdivision}{true}
	\begin{subfigure}{0.4\textwidth}
		\begin{tikzpicture}[scale=3]

	\tikzset{
	-|>-/.style={
	decoration={
			markings,
			mark=at position #1 with {\arrow{latex}}
		},
	postaction={decorate}},
	-|>-/.default=.5
	}

	\tikzset{
	->-/.style={
	decoration={
			markings,
			mark=at position #1 with {\arrow{>}}
		},
	postaction={decorate}},
	->-/.default=.53
	}

	\coordinate (A) at (0,0,0);
	\coordinate (B) at (1,0,0);
	\coordinate (C) at (1,0,1);
	\coordinate (D) at (0,0,1);
	\coordinate (E) at (0,1,0);
	\coordinate (F) at (1,1,0);
	\coordinate (G) at (1,1,1);
	\coordinate (H) at (0,1,1);

	\ifsubdivision\else
	\fi

	\ifsubdivision\else
		\node[above left] at (A) {$0$};
		\node[above right] at (B) {$1$};
		\node[below right] at (C) {$0$};
		\node[below left] at (D) {$1$};
		\node[above left] at (E) {$1$};
		\node[above right] at (F) {$2$};
		\node[above left] at (G) {$1$};
		\node[above left] at (H) {$2$};
	\fi

	\draw[->-] (A) -- (B);
	\draw[->-] (A) -- (D);
	\draw[->-] (C) -- (B);
	\draw[->-] (C) -- (D);

	\draw[->-] (A) -- (E);
	\draw[->-] (B) -- (F);
	\draw[->-] (C) -- (G);
	\draw[->-] (D) -- (H);

	\draw[->-] (E) -- (F);
	\draw[->-] (E) -- (H);
	\draw[->-] (G) -- (F);
	\draw[->-] (G) -- (H);

	\ifsubdivision
		\foreach\x in {0,1} {
				\foreach\y in {0,1} {
						\foreach\z in {0,1} {
								\draw[dotted]
								(\x,\y,\z) -- (\x, \y, 0.5) -- (\x, 0.5, 0.5) -- cycle
								(\x,\y,\z) -- (\x, 0.5, \z) -- (\x, 0.5, 0.5) -- cycle
								(\x,\y,\z) -- (\x, 0.5, \z) -- (0.5, 0.5, \z) -- cycle
								(\x,\y,\z) -- (0.5, \y, \z) -- (0.5, 0.5, \z) -- cycle
								(\x,\y,\z) -- (\x, \y, 0.5) -- (0.5, \y, 0.5) -- cycle
								(\x,\y,\z) -- (0.5, \y, \z) -- (0.5, \y, 0.5) -- cycle
								;
							}
					}
			}
		\draw[red,thick]
		(0,0.5,0.5) -- (0,0.5,0) -- (0.5,0.5,0)
		(1,0.5,0.5) -- (1,0.5,1) -- (0.5,0.5,1);

		\draw[red, dashed]
		(0,0.5,0.5) -- (0,0.5,0) -- (0.5,0.5,0) -- (0.5,0,0) -- (0.5,0,0.5) -- (0,0,0.5) -- cycle
		(1,0.5,0.5) -- (1,0.5,1) -- (0.5,0.5,1) -- (0.5,0,1) -- (0.5,0,0.5) -- (1,0,0.5) -- cycle
		;

		\draw[red,fill] (0.5,0,0.5) circle (0.01);

		\path[pattern=north west lines, pattern color=red] (0,0,0.5) -- (0,0.5,0.5) -- (0,0.5,0) -- (0,0,0) -- cycle;
		\path[pattern=north west lines, pattern color=red] (0,0.5,0) -- (0.5,0.5,0) -- (0.5,0,0) -- (0,0,0) -- cycle;
		\path[pattern=north west lines, pattern color=red] (0.5,0,0) -- (0.5,0,0.5) -- (0,0,0.5) -- (0,0,0) -- cycle;

		\path[pattern=north west lines, pattern color=red] (1,0,0.5) -- (1,0.5,0.5) -- (1,0.5,1) -- (1,0,1) -- cycle;
		\path[pattern=north west lines, pattern color=red] (1,0.5,1) -- (0.5,0.5,1) -- (0.5,0,1) -- (1,0,1) -- cycle;
		\path[pattern=north east lines, pattern color=red] (0.5,0,1) -- (0.5,0,0.5) -- (1,0,0.5) -- (1,0,1) -- cycle;

		\draw[blue]
		(0,0,1) -- (0,1,1)
		(1,0,0) -- (1,1,0)
		;

		\draw[blue, pattern=north west lines, pattern color=blue]
		(1,1,0) -- (1,1,1) -- (0,1,1) -- (0,1,0) -- cycle;
		\draw[blue,fill] (0.5,1,0.5) circle (0.01);
		\draw[blue, thick] (0,0.5,1) -- (0,1,1) -- (1,1,0) -- (1,0.5,0);
	\fi

\end{tikzpicture}
	\end{subfigure}
	\caption{On the left, the product of a bad square with an interval. On the right, the ascending (blue) and descending (red) face links. The descending link is a join of the face links of a bad square and of an interval, while the ascending link only collapses on a join (drawn with a thick blue line).}
	\label{fig:bad-square-times-interval}
\end{figure}

\begin{lemma}\label{res:face-link-product}
	Let $\cube$ be a cube of $\cubulation$, and let $ \cube[1] \times \dots \times \cube[\ell] $ be its decomposition into monochromatic factors. The descending face link of $\cube$ is a subdivision of the join of the descending links of the $\cube[i]$.
\end{lemma}

\begin{proof}
	To see clearly the join decomposition, it is more convenient to look at the dual of the cubes, which are cross-polytopes. This way, we have that $ \boundary \cube^* $ is the join of the $ \boundary \cube[i]^* $. Recall also that the barycentric subdivision of a polytope is isomorphic to the barycentric subdivision of its dual.

	By the definition of $\baryf$, the barycentre of some face $\varcube$ of $\cube$ belongs to the descending face link of $\cube$ if and only if $\varcube$ contains a minimum for $ \lift \baryf \colon \cube \to \RR $. If $\varcube$ decomposes as a product of $ \varcube[1] \times \dots \times \varcube[\ell] $, requiring that $\varcube$ contains a minimum for $ \lift \baryf $ is equivalent to requiring that $\varcube[i]$ contains a minimum for $ \lift f_i$ for all $i$, which is in turn equivalent to the barycentre of $\varcube[i]$ belonging to the descending face link of $\cube[i]$. By looking dually, this proves that the descending face link of $\cube$ is the join of the descending face links of the $\cube[i]$.

	Similarly, the barycentre of $ \varcube $ belongs to the ascending link if and only if at least one of the barycentres of the $ \varcube[i] $ belong to the ascending link of $ \cube[i] $. This means that the ascending link is a regular neighbourhood of the join of the ascending links of the $ \cube[i] $, and so it collapses on that join (for details on the definition of regular neighbourhood, we refer to \cite[Chapter 4]{RourkeSanderson}).
\end{proof}

\begin{proposition}\label{cor:face-link-good-cube-and-bad-squares}
	If $ \cube $ is a good cube, then the ascending and descending face links of $\cube$ are collapsible.

	If $\cube$ is a $2k$-cube, and every move contains either zero or two facets of $\deffacets$, then the ascending and descending face links of $\cube$ collapse to the barycentric subdivision of the boundary of a $k$-dimensional cross-polytope (which is PL-homeomorphic to $ S^{k-1} $).
\end{proposition}

\begin{proof}
	In the first case, one of the factors of $\cube$ is an interval, whose ascending and descending face link is a point. By \cref{res:face-link-product}, the ascending and descending links of $\cube$ collapse on a cone, which is in turn collapsible.

	In the second case, all the monochromatic factors of $\cube$ are squares, so by \cref{lemma:face-link-monochromatic} the ascending and descending link collapse on the join of $k$ copies of $S^0$, i.e.~the boundary of a $k$-dimensional cross-polytope.
\end{proof}

\begin{remark}
	For vertices of the cubulation, the face link is empty, and in particular non-collapsible. This does not contradict \cref{cor:face-link-good-cube-and-bad-squares} as vertices are \emph{bad} by definition.
\end{remark}

\begin{corollary}\label{res:good-cube-collapsible-face-link}
	The ascending and descending links of $ \baryf $ at the barycentre of a good cube are collapsible.
\end{corollary}

\begin{proof}
	Follows by combining \cref{cor:face-link-good-cube-and-bad-squares} with \cref{lemma:linkup-join-flinkup-coflinkup}.
\end{proof}

\subsection{The coface link}

Let us switch our focus to the coface link. Let $\cube$ be a cube of $ \cubulation $; we are mostly interested in the case when $ \cube $ is bad, otherwise we already know that the ascending and descending links at $ \barycentre \cube $ are collapsible.

\begin{remark}
	To better understand this section, it is convenient to keep in mind the simple case where $\cube$ is $0$-dimensional, i.e.~it is a vertex $v$. In this case, $\cube$ is the barycentre of $\pol[v]$, and the coface link is the whole link.
\end{remark}

We employ the same notation used for the study of the face link, so we let $\cube$ be dual to some cell $\face$ of the tessellation of $\mfld$ into copies of $\pol$.
Let $v$ be a vertex of $ \cube $, so that $ \pol[v] $ contains $ \face $, and let $ \facet[1], \dots, \facet[k] $ denote the facets of $\pol[v]$ that intersect in $ \face $.

The coface link of $\cube$, being the dual notion of the face link, is the barycentric subdivision of the boundary of $ \face $. Its vertices are in correspondence with all the proper cofaces of $\cube$; to compute the ascending and descending coface link, a priori we would need to compute for each coface of $\cube$ whether the corresponding barycentre is in the ascending or descending link. The following lemma allows us to restrict our analysis to cofacets, i.e.~cofaces of dimension exactly one higher than $ \dim \cube $.

\begin{proposition}\label{prop:visual-link}
	Let $ \cube \subset \varcube $ be two cubes of $\cubulation$.
	Then $ \barycentre\varcube \in \coflinkup Q$ if and only if for every cofacet $\varvarcube$ of $\cube$, with $\varvarcube \subseteq \varcube $, we have $ \barycentre\varvarcube \in \coflinkup \cube $.
	Conversely, we have that $ \barycentre{Q'} \in \coflinkdown Q $ if and only if there is some cofacet of $ Q $ contained in $ Q' $ whose barycentre is in $ \coflinkdown Q $.
\end{proposition}

\begin{proof}
	Lift $f \colon \varcube \to S^1$ to a function $ \lift f \colon \varcube \to \RR $ such that its minimum is $0$. Then $ \lift f(\barycentre\varcube) = \epsilon \cdot \dim\varcube  $, and $ \lift f (\barycentre\cube) < \lift f(\barycentre\varcube) $ if and only if $ \lift f $ assumes the value $0$ on $\cube$. In this case, for every cofacet $ \cube \subset \varvarcube \subseteq \varcube$ we have that $ \lift f(\barycentre\varvarcube) = \epsilon \cdot \dim \varvarcube > \lift f(\barycentre\cube)$, so all cofacets of $\cube$ contained in $\varcube$ have their barycentre in $ \coflinkup \cube $.

	Vice versa, decompose $\varcube$ as a product $\varcube[1] \times \dots \times \varcube[\ell]$ of its monochromatic factors, and define $ \cube[i]= \varcube[i] \cap \cube $, so that $ \cube = \cube[1] \times \dots \times \cube[i] $.

	Suppose that $ \lift f (\barycentre \cube) > \lift f (\barycentre\varcube)  $. Then by using \cref{sum-decomposition} there must be a $\cube[i]$ with $ \lift f (\barycentre{\cube[i]}) > \lift f(\barycentre{\varcube}) $; so $ \cube[i] $ does not contain a minimum of $ \lift f \colon  \varcube[i] \to \RR $. The only subcubes of a monochromatic cube that do not contain a minimum are the vertices where $ \lift f $ attains maximum. Therefore, the cube $ \cube[i] $ is $0$-dimensional.
	If we let $\varvarcube[i]$ be any edge such that $ \cube[i] \subset \varvarcube[i] \subseteq \varcube[i] $, we obtain a cofacet
	\[
		\varvarcube = \cube[1] \times \dots \times  \varvarcube[i] \times \dots \times \cube[\ell]
	\]
	of $Q$, whose barycentre is contained in $ \coflinkdown \cube $, as requested.
\end{proof}

Cofacets of $\cube$ correspond dually to facets of $\face$. If $ \face = \facet[1] \cap \dots \cap \facet[k] $, then the facets of $ \face $ are in natural correspondence with the facets $\facet$ of $\pol$ that are adjacent to all $F_i$, for all $ i \in \range k $; the correspondence is given by $ \facet \mapsto \facet \cap \face $.
Recall that $ \face $ inherits a state from $ \pol $, by assigning to the facet $ \facet \cap \face $ of $ \face $ an inherited status as follows:
\begin{itemize}
	\item if $\facet$ is in the same move as at least one of the $ \facet[i] $, then we assign status $ \Out $;
	\item else, we assign the same status as $ \facet $.
\end{itemize}

We are finally able to motivate this definition with the following lemma.

\begin{lemma}\label{res:state-inheritance}
	Let $\cube$ be a $k$-cube of $\cubulation$, and let $v$ be one of its vertices. Consider the copy $\pol[v]$ dual to $v$, and let $ \face $ be the face of $\pol[v]$ dual to $\cube$.

	Let $\varcube$ be a cofacet of $\cube$, and let $\facet$ be the facet of $\pol[v]$ such that $\varcube$ is dual to $ \facet \cap \face$.

	Then $ \barycentre\varcube \in \coflinkup \cube $ if and only if $ \facet \cap \face $ has inherited status $\Out$ in $\face$.
\end{lemma}
\begin{figure}
	\centering
	\setboolean{bad}{false}
	\begin{subfigure}{0.35\textwidth}
		\begin{tikzpicture}[scale=2.5]
	\usetikzlibrary{positioning}
	\usetikzlibrary{calc}
	\usetikzlibrary{backgrounds,fit}
	\usetikzlibrary{patterns}

	\tikzset{
	-|>-/.style={
	decoration={
			markings,
			mark=at position #1 with {\arrow{latex}}
		},
	postaction={decorate}},
	-|>-/.default=.60
	}

	\tikzset{
	->-/.style={
	decoration={
			markings,
			mark=at position #1 with {\arrow{>}}
		},
	postaction={decorate}},
	->-/.default=.53
	}

	\coordinate (A) at (0,0);
	\coordinate (B) at (1,0);
	\coordinate (C) at (1,1);
	\coordinate (D) at (0,1);

	\tiny
	\node[below left] at (A) {$1$};
	\ifbad
		\node[below right] at (B) {$0$};
	\else
		\node[below right] at (B) {$2$};
	\fi
	\node[above right] at (C) {$1$};
	\node[above left] at (D) {$0$};

	\coordinate (E) at (0.5, 0);
	\coordinate (F) at (1, 0.5);
	\coordinate (G) at (0.5, 1);
	\coordinate (H) at (0, 0.5);
	\coordinate (I) at (0.5, 0.5);

	\ifbad
		\node[below right] at (E) {$\epsilon$};
		\node[above right] at (F) {$\epsilon$};
	\else
		\node[below right] at (E) {$1+\epsilon$};
		\node[above right] at (F) {$1+\epsilon$};
	\fi
	\node[above right] at (G) {$\epsilon$};
	\node[above left] at (H) {$\epsilon$};
	\node[below] at (0.3,0) {$\cube$};
	\node at (0.67,0.8) {$\varcube$};
	\node[above left] at (A) {$v$};
	\node[below left] at (D) {$v'$};

	\node[above right=-0.05 and .1] at (I) {$2\epsilon$};

	\ifbad
		\draw[->-] (B) -- (E);
		\draw[->-] (E) -- (A);
	\else
		\draw[->-] (A) -- (E);
		\draw[->-] (E) -- (B);
	\fi
	\draw[->-,blue] (H) -- (A);
	\draw[->-, blue] (D) -- (H);

	\ifbad
		\draw[->-] (F) -- (C);
		\draw[->-] (B) -- (F);
	\else
		\draw[->-] (F) -- (B);
		\draw[->-] (C) -- (F);
	\fi
	\draw[->-] (D) -- (G);
	\draw[->-] (G) -- (C);

	\ifbad
		\draw[->-,red] (E) -- (I);
		\draw[->-] (F) -- (I);
	\else
		\draw[->-,red] (I) -- (E);
		\draw[->-] (I) -- (F);
	\fi
	\draw[->-] (H) -- (I);
	\draw[->-] (G) -- (I);

	\draw[->-] (I) -- (A);
	\draw[->-] (I) -- (C);
	\draw[->-] (D) -- (I);
	\ifbad
		\draw[->-] (B) -- (I);
	\else
		\draw[->-] (I) -- (B);
	\fi

	\draw[dashed]
	(E) -- ++(0,-.4)
	(F) -- ++(.2,0)
	(G) -- ++(0,+.2)
	(H) -- ++(-.4,0)
	;

	\draw[red,fill] (E) circle (0.03);
	\draw[blue,fill] (A) circle (0.03);
	\draw[fill] (I) circle (0.02);

	\node at (-.3,-.3) {$P_v$};
	\node[right] at (.5,-.3) {$\face$};
	\node[above] at (-.3,.5) {$F$};
	\node[below right,fill=white,text opacity=1,fill opacity=.9,rounded corners=10pt] at (I) {$\varfacet$};

\end{tikzpicture}
	\end{subfigure}
	\setboolean{bad}{true}
	\begin{subfigure}{0.35\textwidth}
		\begin{tikzpicture}[scale=2.5]
	\usetikzlibrary{positioning}
	\usetikzlibrary{calc}
	\usetikzlibrary{backgrounds,fit}
	\usetikzlibrary{patterns}

	\tikzset{
	-|>-/.style={
	decoration={
			markings,
			mark=at position #1 with {\arrow{latex}}
		},
	postaction={decorate}},
	-|>-/.default=.60
	}

	\tikzset{
	->-/.style={
	decoration={
			markings,
			mark=at position #1 with {\arrow{>}}
		},
	postaction={decorate}},
	->-/.default=.53
	}

	\coordinate (A) at (0,0);
	\coordinate (B) at (1,0);
	\coordinate (C) at (1,1);
	\coordinate (D) at (0,1);

	\tiny
	\node[below left] at (A) {$1$};
	\ifbad
		\node[below right] at (B) {$0$};
	\else
		\node[below right] at (B) {$2$};
	\fi
	\node[above right] at (C) {$1$};
	\node[above left] at (D) {$0$};

	\coordinate (E) at (0.5, 0);
	\coordinate (F) at (1, 0.5);
	\coordinate (G) at (0.5, 1);
	\coordinate (H) at (0, 0.5);
	\coordinate (I) at (0.5, 0.5);

	\ifbad
		\node[below right] at (E) {$\epsilon$};
		\node[above right] at (F) {$\epsilon$};
	\else
		\node[below right] at (E) {$1+\epsilon$};
		\node[above right] at (F) {$1+\epsilon$};
	\fi
	\node[above right] at (G) {$\epsilon$};
	\node[above left] at (H) {$\epsilon$};
	\node[below] at (0.3,0) {$\cube$};
	\node at (0.67,0.8) {$\varcube$};
	\node[above left] at (A) {$v$};
	\node[below left] at (D) {$v'$};

	\node[above right=-0.05 and .1] at (I) {$2\epsilon$};

	\ifbad
		\draw[->-] (B) -- (E);
		\draw[->-] (E) -- (A);
	\else
		\draw[->-] (A) -- (E);
		\draw[->-] (E) -- (B);
	\fi
	\draw[->-,blue] (H) -- (A);
	\draw[->-, blue] (D) -- (H);

	\ifbad
		\draw[->-] (F) -- (C);
		\draw[->-] (B) -- (F);
	\else
		\draw[->-] (F) -- (B);
		\draw[->-] (C) -- (F);
	\fi
	\draw[->-] (D) -- (G);
	\draw[->-] (G) -- (C);

	\ifbad
		\draw[->-,red] (E) -- (I);
		\draw[->-] (F) -- (I);
	\else
		\draw[->-,red] (I) -- (E);
		\draw[->-] (I) -- (F);
	\fi
	\draw[->-] (H) -- (I);
	\draw[->-] (G) -- (I);

	\draw[->-] (I) -- (A);
	\draw[->-] (I) -- (C);
	\draw[->-] (D) -- (I);
	\ifbad
		\draw[->-] (B) -- (I);
	\else
		\draw[->-] (I) -- (B);
	\fi

	\draw[dashed]
	(E) -- ++(0,-.4)
	(F) -- ++(.2,0)
	(G) -- ++(0,+.2)
	(H) -- ++(-.4,0)
	;

	\draw[red,fill] (E) circle (0.03);
	\draw[blue,fill] (A) circle (0.03);
	\draw[fill] (I) circle (0.02);

	\node at (-.3,-.3) {$P_v$};
	\node[right] at (.5,-.3) {$\face$};
	\node[above] at (-.3,.5) {$F$};
	\node[below right,fill=white,text opacity=1,fill opacity=.9,rounded corners=10pt] at (I) {$\varfacet$};

\end{tikzpicture}
	\end{subfigure}
	\caption{In this picture, we represented a cube $\varcube$ of the cubulation, which is a cofacet of another cube $Q$ (the bottom edge). The blue vertex is the barycentre of $\pol[v]$, and the blue edge is dual to a facet $\facet$ of $\pol[v]$ with status $\In$. When computing the coface link at the red vertex, which is the barycentre of a face $\face$ of $\pol[v]$, we need to know the orientation of the red edge, corresponding to the status of the facet $\varfacet=\facet \cap \face$ of $\face$.
		In the picture on the left, the square is coherent, so the orientation of the red edge coincides with the orientation of the blue one: the status of $\varfacet$ coincides with the status of $\facet$. Vice versa, on the right, we have a bad square: in this case, the red edge is always oriented outwards, so the status of $\varfacet$ is always $\Out$, independently of the status of $\facet$. }
	\label{fig:coface-link-status}
\end{figure}

\begin{proof}
	Consider the lift $ \lift f: Q' \to \RR $, chosen such that the minimum is $0$, and denote by $ v' $ the only vertex of $ \varcube $ adjacent to $v$ but not belonging to $ \cube $. By definition, $ \lift f(v') > \lift f(v) $ if and only if the status of $\facet$ in $ \pol $ is $ \Out $.

	Let $ \varcube = \varcube[1] \times \dots \times \varcube[\ell] $ be the decomposition of $Q$ into monochromatic factors, and let $ \cube[i] = \varcube[i] \cap \cube $. We may assume that $ \cube[1] $ has codimension $1$ in $ \varcube[1] $, and that $ \varcube[i]=\cube[i] $ for $ i \geq 2 $. Let $ v_1, v_1' $ be the projections of $ v, v' $ to $ \varcube[1] $.

	Let $ \face = \facet[1] \cap \dots \cap \facet[k] $, where $ \facet[i] $ is a facet of $ \pol[v] $.
	If $\facet$ is in the same move as some $ \facet[i] $, then $ \cube[1] $ is positive dimensional, and therefore contains a minimum of $ \restrict{\lift f}{\varcube[1]} $. This implies that $ \lift f(\barycentre {\cube[1]}) = \epsilon \dim \cube [1] $ and $ \lift f(\barycentre{\varcube[1]}) = \epsilon \dim \varcube[1] $. So $ \barycentre\varcube \in \coflinkup \cube $, and by definition the inherited status is always $ \Out $, independent of the status of $\facet$ in $ \pol $.

	Suppose now that $ \facet $ is in a different move from all the $ \facet[i] $. This means that $ \cube[1] $ is the vertex $ v$, and $ \varcube[1] $ is an edge connecting $ v $ and $ v' $. In this case, it follows that $ \lift f(\barycentre{\varcube[1]})>\lift f(\barycentre {\cube[1]}) $ if and only if $ \lift f(v') > \lift f (v) $. Since $ \lift f $ splits as a sum on all the monochromatic factors, this implies that $ \barycentre \varcube \in \coflinkup \cube $ if and only if the status of $ \facet $ in $ \pol $ is $ \Out $, as desired.
\end{proof}

Combining \cref{prop:visual-link,res:state-inheritance} we are able to compute the coface link from the inherited state.
More precisely, let $v$ be a vertex of the cubulation, and let $ \face $ be a face of $ \pol[v] $, dual to some cube $ \cube $ of $ \cubulation $. The dual polytope of $ \face $, denoted by $\dualface$, is simplicial (since $ \face $ is right-angled), and its vertices inherit a labelling from the inherited state on $ \face $. As we did for $ \pol $, we consider the full subcomplexes $ \dualface_\In $, $ \dualface_\Out $ generated by vertices with label $ \In $ and $ \Out $ respectively. The coface link is naturally isomorphic to the barycentric subdivision of the boundary of $ \dualface $. We have the following.

\begin{proposition}\label{res:legal-inherited-collapsible-coface-link}
	The ascending coface link of $Q$ is the barycentric subdivision of $ \dualface_\Out $. The descending coface link is the regular neighbourhood of $ \dualface_\In $ in $ \boundary\dualface $, that is the full subcomplex of the barycentric subdivision whose vertices are the barycentres of all the simplices that intersect $ \dualface_\In $.
\end{proposition}

\begin{proof}
	Fix a simplex $ \sigma $ of the boundary of $ \dualface $; this corresponds to some cube $ \varcube \supset \cube $ in the cubulation. If all the vertices of $ \sigma $ have inherited status $\Out$, by \cref{res:state-inheritance} all cofacets $ \cube \subset \varvarcube \subseteq \varcube $ have their barycentres in the ascending coface link, and by \cref{prop:visual-link} this implies that $ \barycentre\varcube $ also belongs to the ascending coface link.

	Vice versa, if one of the vertices of $ \sigma $ has status $ \In $, then we can find a cofacet $\cube \subset \varvarcube \subseteq \varcube $ whose barycentre is in the descending coface link, and therefore by \cref{prop:visual-link} we have that $ \barycentre\varcube \in \coflinkdown\cube$.
\end{proof}

\begin{corollary}\label{res:collapsible-inherited-state}
	If the inherited state on $ \face $ is $m$-legal, then both $ \coflinkup \cube $ and $ \coflinkdown \cube $ are $ m $-connected. If it is $ \infty $-legal, they are collapsible.
\end{corollary}

\begin{proof}
	By \cref{res:legal-inherited-collapsible-coface-link}, we know that $ \coflinkup \cube $ and $ \coflinkdown \cube $ collapse on the barycentric subdivision of $ \dualface_\Out $ and $ \dualface_\In $ respectively. If the inherited state is $m$-legal (resp.~$ \infty $-legal), then those are by definition $ m $-connected (resp.~collapsible), and so are $ \coflinkup \cube $ and $ \coflinkdown \cube $.
\end{proof}

We conclude by proving \cref{thm:inherited-state-collapsibility}.

\begin{proof}[Proof of \cref{thm:inherited-state-collapsibility}]
	We prove that the map $ f \colon \cubulation \to S^1 $ constructed in this section satisfies the requirements of the theorem. Indeed, consider a face $ \face[v] $ of $ \pol[v] $, dual to some cube $ \cube $ of $\cubulation$.

	If $ \face[v] $ is good, then by \cref{res:good-cube-collapsible-face-link} both $ \flinkup \cube $ and $ \flinkdown \cube $ are collapsible, and therefore also $ \linkup {\barycentre \cube} $ and $ \linkdown {\barycentre \cube} $ are.

	If the inherited state $\inheritedstate[v] $ on $ \face[v] $ is $ m $-legal (resp.~totally legal), then by \cref{res:collapsible-inherited-state} both $ \coflinkup \cube $ and $ \coflinkdown \cube $ are $m$-connected (resp.~collapsible), and so are $ \linkup {\barycentre \cube} $ and $ \linkdown {\barycentre \cube} $.
\end{proof}

\def\polytopename{P^6}
\def\truncpolytopename{\bar P^6}

\section{The polytope $P^6$}\label{sec:P6}
In this section, we introduce the $6$-dimensional polytope $P^6$ to which the L\"obell construction will be applied, and describe its combinatorial structure. In the latter part of the section, we present an initial state, and the set of moves that allow us to construct the Bestvina-Brady function using the theory described in Section \ref{sec:bary-sub}.

The polytope $P^6$ belongs to a particular sequence of polytopes, and shares many properties with the other members of its family. For this reason, we start by describing this sequence, in a wider context.

\subsection{Potyagailo-Vinberg polytopes}
There is a remarkable sequence of right-angled hyperbolic polytopes $P^n \subset \Hyp n$, for $n=3, \dots, 8$, which was first considered by Agol, Long, and Reid \cite{AgolLongReid}. They are obtained by considering the star of a vertex in a tessellation of $\HH^n$ by copies of a suitable Coxeter simplex. The combinatorics of these polytopes was later studied by Potyagailo and Vinberg \cite{PotyagailoVinberg}, and Everitt, Ratcliffe, and Tschantz \cite{EverittRatcliffeTschantz}; in particular, it was shown that they are dual to some semi-regular Euclidean polytopes discovered by Gosset \cite{Gosset} in 1900. These polytopes were already used in \cite{IMMAlgebraicFibering}, to construct algebraic fibring of some hyperbolic manifolds, and in particular the polytope $P^5$ is the one used in \cite{IMMHyperbolic5Manifold} to construct a hyperbolic $5$-manifold fibring over the circle.

We summarize in the next paragraphs the properties of the polytopes that are needed for this paper, and we refer to the previous sources \cites{EverittRatcliffeTschantz, PotyagailoVinberg} for more detailed information about the polytopes, and \cite{IMMHyperbolic5Manifold} for the particular use in this context.

It is worth noting that each $k$-dimensional face of $P^n$ is isometric to the polytope $P^k$, so these polytopes are closely related to each other. For this reason, even though the main protagonist of our construction is $P^6$, we start our analysis from the $5$-dimensional polytope $P^5$.

\subsection{The polytope $P^5$}
The polytope $P^5$ has $16$ facets, $16$ real vertices and $10$ ideal vertices. Each facet is isometric to the polytope $P^4$, and each real vertex is opposed to a facet. It is dual to the Gosset polytope $1_{21}$.

It is possible to visualize $P^5$ inside the Klein model for the hyperbolic space $\HH^5$, as the intersection of the $16$ half-spaces $$\pm x_1 \pm x_2 \pm x_3 \pm x_4 \pm x_5\leq 1$$ having an even number of minus signs. The $10$ ideal vertices are precisely the points
$$(\pm 1,0,0,0,0), \ (0, \pm 1,0,0,0), \ (0,0,\pm 1,0,0),  \ (0,0,0,\pm 1,0), \ (0,0,0,0,\pm 1);$$
while the real vertices are the points of the form $\left(\pm \frac13, \pm \frac13, \pm \frac13, \pm \frac13, \pm \frac13\right)$ with an odd number of minus signs.
Each real vertex is opposed a facet.

In this simple description, all its isometries are the maps of the form
$$(x_1, x_2, x_3, x_4, x_5) \mapsto (\pm x_{\sigma(1)}, \pm x_{\sigma(2)}, \pm x_{\sigma(3)}, \pm x_{\sigma(4)}, \pm x_{\sigma(5)})$$
where $\sigma\in S_5$ is any permutation, and the number of minus signs is even. The isometry group $\Isom(P^5)$ has therefore order $2^4\cdot 5!=1920$.

Since there is an even number of minus signs, each isometry is determined by the first four coordinates of the images. We look at $\RR^5$ as $\RR^4\times \RR$, and interpret $\RR^4$ as the quaternion space $ \HH $. This point of view allows us to better describe two useful subgroups of $\Isom(P^5)$, already considered in \cite{IMMHyperbolic5Manifold}, that are $Q_8$ and $R_{16}$. Here we summarize their definitions and properties.

The group $Q_8$ is the quaternion group $\{\pm 1, \pm i, \pm j, \pm k\}$. It acts on the quaternion space $\HH=\RR^4$ by left multiplication, and this action can be extended trivially on $\RR^4\times \RR$. Since it preserves all the half-spaces defining $P^5$, it preserves the polytope, so it can be seen as a subgroup of $\Isom(P^5)$.

Among the isometries of $P^5$, we consider the involution $\iota$ given by
$$\iota \colon (x_1, x_2, x_3, x_4, x_5) \mapsto (x_1, -x_2, -x_4, -x_3, -x_5)$$
One may verify that $\iota$ normalizes the subgroup $Q_8$, and that the group $R_{16}$ generated by $Q_8$ and $\iota$ has order $16$. Moreover, $R_{16}$ acts transitively on the facets of $P^5$.

Finally, we are interested the binary tetrahedral group $Q_8 \cup \{ \frac{1}{2}(\pm 1\pm i\pm j\pm k) \}$, which has $24$ elements. While it could be seen as a subgroup of $\Isom(P^5)$ in the same way $Q_8$ can, we will not need it in this paper. In fact, it will be just used as a set of labels, to encode some combinatorial properties of the polytopes; to simplify the notation, we will use $T_{24}=Q_8 \cup \{(\pm 1\pm i\pm j\pm k) \}$, omitting the $\frac12$. The action of $R_{16}$ on $ \HH \times \RR $ restricts naturally to an action on $ T_{24} \subset \HH \subset \HH \times \RR $.

\subsection{The polytope $P^6$}
The polytope $P^6$ has $27$ facets (each isometric to $P^5$), $72$ real vertices and $27$ ideal vertices. It is dual to the Gosset polytope $\gosset$. Both $P^6$ and $\gosset$ have many truly remarkable combinatorial properties and symmetries, which we are going to exploit in this paper. For example, the $1$-skeleton of $\gosset$ is the configuration graph of the $27$ lines in a general cubic surface, see \cite{Coxeter27rette}.

The combinatorial structure of $P^6$ is described in \cite{EverittRatcliffeTschantz} by associating a vector of $\RR^7$ to each facet. These vectors represent the unitary (with respect to the Lorentzian metric) normal vector to each facet, in one particular embedding of $P^6$ into the hyperboloid model of $\HH^6$. The vectors are listed in \cref{table:G6}. The convex hull of the $27$ points in $\RR^7$ represented by these vectors is the Gosset polytope $\gosset$, dual to $P^6$.

From these vectors, one can read the adjacency graph of the facets (or equivalently, the $1$-skeleton of $\gosset$) in the following way: two facets are adjacent if and only if the Lorentzian scalar product of their associated vector is zero.

We now present a labelling of the facets of $P^6$, which better highlights the symmetries coming from $P^5$. The $27$ facets will be labelled by using the group $T_{24}$, together with the symbols $A$, $B$ and $C$. %

We assign a $1-1$ correspondence between the $27$ labels and the $27$ facets, such that the following properties hold:

\begin{itemize}
	\item Facets with label $ A,B,C $ are pairwise disjoint.
	\item Two facets with label in $ T_{24} $ are adjacent if and only if their Euclidean scalar product is non-negative.
	\item The facet $A$ is adjacent to exactly all facets with label $ \pm 1 \pm i \pm j \pm k $.
	\item The facet $B$ (resp.~$C$) is adjacent to all the facets with label in $Q_8$ and to all the facets whose label $ \pm 1 \pm i \pm j \pm k $ has an even (resp.~odd) number of minus signs.
\end{itemize}

In \cref{table:G6} we provide the explicit mapping between this labelling, and the description of the facets of $P^6$ given in \cite{EverittRatcliffeTschantz} using vectors.
\smallskip

The heuristic idea behind the $T_{24}$-labelling is the following. We start by choosing any facet of $P_6$, and label it with $A$. There are $16$ facets adjacent to $A$, in the shape of a $P^5$; we use elements of the type $ \pm 1 \pm i \pm j \pm k $ to label all of them. The correspondence we chose stems from the fact that each facet of $P^5$ corresponds to a hyperplane $\pm x_1 \pm x_2 \pm x_3 \pm x_4 \pm x_5 = 1$ with an even number of minus signs, and is therefore determined by its first four signs. These four signs are used as a choice for the label $ \pm 1 \pm i \pm j \pm k $.

Each of the remaining $10$ facets of $P^6$ is adjacent to $8$ facets adjacent to $A$, and to $8$ facets far from $A$. With a careful look, one can notice that the adjacency pattern of these facets is related to the arrangement of the $10$ ideal vertices of $P^5$. Indeed, one can associate an ideal vertex of $P^5$ to each of these $10$ facets in such a way that a facet close to $A$ is adjacent to a facet far from $A$ if and only if its corresponding facet of $P^5$ is incident to the corresponding ideal vertex of $P^5$. Moreover, adjacency between facets far from $A$ can be read as adjacency of the corresponding ideal vertices in the $1$-skeleton of $P^5$.
The ideal vertices of $P^5$ correspond to the vectors of the canonical basis of $\RR^5$ and their opposites. The eight vectors that have fifth coordinate equal to zero get the label in $Q_8$ representing the quaternion in the identification $\RR^4=\HH$, while the vertices corresponding to the vectors $(0,0,0,0,\pm 1)$ get labels $B$ and $C$.

From now on, we will always refer to facets of $P_6$ by using their label.
\smallskip

\begin{table}
	\centering

	\begin{tabular}{C|C||C|C}
		\textrm{Vector}  & \textrm{Label}                    & \textrm{Vector}  & \textrm{Label}       \\
		\hline \hline    & \vspace{-10pt}                    &                  &                      \\
		(0,0,0,0,0,-1,0) & A                                 &                  &                      \\
		\hline           & \vspace{-10pt}                    &                  &                      \\
		(0,0,0,0,-1,0,0) & \hphantom{-} 1+i+j+k \hphantom{-} & (1,1,1,1,1,0,2)  & -1-i-j-k             \\
		(1,1,0,0,0,0,1)  & \hphantom{-} 1+i-j-k \hphantom{-} & (1,0,1,0,0,0,1)  & \hphantom{-}1-i+j-k  \\
		(1,0,0,1,0,0,1)  & \hphantom{-} 1-i-j+k \hphantom{-} & (1,0,0,0,1,0,1)  & \hphantom{-}1-i-j-k  \\
		(0,1,1,0,0,0,1)  & -1+i+j-k             \hphantom{-} & (0,1,0,1,0,0,1)  & -1+i-j+k             \\
		(0,1,0,0,1,0,1)  & -1+i-j-k             \hphantom{-} & (0,0,1,1,0,0,1)  & -1-i+j+k             \\
		(0,0,1,0,1,0,1)  & -1-i+j-k             \hphantom{-} & (0,0,0,1,1,0,1)  & -1-i-j+k             \\
		(-1,0,0,0,0,0,0) & -1+i+j+k             \hphantom{-} & (0,-1,0,0,0,0,0) & \hphantom{-}1-i+j+k  \\
		(0,0,-1,0,0,0,0) & \hphantom{-} 1+i-j+k \hphantom{-} & (0,0,0,-1,0,0,0) & \hphantom{-} 1+i+j-k \\
		\hline           & \vspace{-10pt}                    &                  &                      \\
		(1,0,0,0,0,1,1)  & 1                                 & (0,1,1,1,1,1,2)  & -1                   \\
		(0,1,0,0,0,1,1)  & i                                 & (1,0,1,1,1,1,2)  & -i                   \\
		(0,0,1,0,0,1,1)  & j                                 & (1,1,0,1,1,1,2)  & -j                   \\
		(0,0,0,1,0,1,1)  & k                                 & (1,1,1,0,1,1,2)  & -k                   \\
		(0,0,0,0,1,1,1)  & C                                 & (1,1,1,1,0,1,2)  & B                    \\
	\end{tabular}
	\caption{The explicit labelling of the facets of $P^6$ using $T_{24}\cup\{A,B,C\}$, with respect to the description with associated vectors.}
	\label{table:G6}
\end{table}

\newcommand{\Rsixteen}{R_{16}}
\newcommand{\bintetra}{T_{24}}
\newcommand{\invol}{\iota}
The group $\Rsixteen$, acting on $P^5$ by isometries, also acts on the facet $A$ (which is indeed a copy of $P^5$), and the isometries extend uniquely to isometries of $P^6$ in a way that is encoded by the $T_{24}$-labelling.

To see this, recall that $\Rsixteen$ acts naturally on the labels in $T_{24}$, as they are elements of $ \HH $. By the way labels are defined, the labelling is $\Rsixteen$-equivariant, meaning that the facet with label $q$ is sent by $ \phi \in \Rsixteen $ to the facet with label $ \phi(q) $. To be precise, we also need to define the action of $\Rsixteen$ on the labels $ A,B,C $: to get equivariance, we let $Q_8\subset \Rsixteen$ act on them trivially, and we let $ \invol $ act by fixing $A$ and swapping $ B $ and $C$. This is coherent with the fact that $\Rsixteen$ acts by isometries on the facet $A$ by definition, so it sends $A$ to itself.

\subsection{The set of moves}

\newcommand{\lab}{r}

To define the moves on the facets with label in $\bintetra$, we consider the action of $Q_8$ by left multiplication: this partitions $\bintetra$ into three orbits. We choose a base point in each: in particular, we pick $ 1 $, $ 1-i+j-k $, and $ 1+i+j-k $. They are pairwise adjacent and they are invariant, as a set, under the involution $ \invol $. Every $ t \in T_{24} $ is of the form $ q \cdot t' $, with $t'$ one of the base points and $ q \in Q_8 $. We define $ \lab\colon T_{24} \to Q_8 $ by setting $ \lab(t) = q $: note that $r$ is \emph{not} a group homomorphism.

The facets with label in $\bintetra$ are then subdivided into four sets of six facets each, that are the preimages of $ \{\pm 1\} , \{\pm i\} , \{\pm j\} , \{\pm k\} $ via $ \lab $; we let each of these four sets define a move. The explicit values of $ \lab(t) $ are computed in \cref{table:r-of-t}.

We define a fifth move to be the subset containing the triple $\set{A,B,C}$. Summing up, we have $4$ moves containing $6$ facets each, and another move containing $3$ disjoint facets.

\begin{table}
	\centering
	\begin{tabular}{RRR|C}
		\multicolumn{3}{C|}{t} & r(t)                                 \\ \hline
		1,                     & 1-i+j-k,  & 1+i+j-k  & \hphantom{-}1 \\
		-1,                    & -1+i-j+k, & -1-i-j+k & -1            \\\hline
		i,                     & 1+i+j+k,  & -1+i+j+k & \hphantom{-}i \\
		-i,                    & -1-i-j-k, & 1-i-j-k  & -i            \\\hline
		j,                     & -1-i+j+k, & -1-i+j-k & \hphantom{-}j \\
		-j,                    & 1+i-j-k,  & 1+i-j+k  & -j            \\\hline
		k,                     & 1-i-j+k,  & 1-i+j+k  & \hphantom{-}k \\
		-k,                    & -1+i+j-k, & -1+i-j-k & -k            \\
	\end{tabular}
	\caption{The values of $ \lab $ on each facet. The facets in each row are pairwise adjacent; the horizontal lines divide the facets into four moves. The three columns form the three orbits of the action of $Q_8$ by left multiplication.}
	\label{table:r-of-t}
\end{table}

\begin{lemma}
	The map $ \lab\colon T_{24} \to Q_8 $ is $\Rsixteen$-equivariant with respect to the natural action of $\Rsixteen$ on $\bintetra$ and $Q_8$.
\end{lemma}

\begin{proof}
	Suppose that some facet has label $ t $ with $ \lab(t)=q' $; be definition $ t=q' \cdot t' $, with $t' \in \{1, 1-i+j-k, 1+i+j-k\} $.

	If $ q \in Q_8 $, then $\lab$ satisfies by definition $ \lab(q \cdot t) = q \cdot \lab(t) $. On the other hand, one can check that $ \invol $ is an automorphism of $\HH$, so we have that $ \invol(t)=\invol(q' \cdot t')= \invol(q') \cdot \invol(t') $, and since the base points are invariant under $ \invol $ we obtain $ r(\invol(t)) = \invol(q') = \invol(r(t))$, as desired.
\end{proof}

\begin{corollary}\label{res:r16-preserves-moves}
	The action of $\Rsixteen$ on $P^6$ by isometries preserves the set of moves as a partition.
\end{corollary}

\begin{proof}
	Follows from the fact that the map $ \lab $ is equivariant, and the action $ \Rsixteen \acts Q_8 $ preserves the partition of $Q_8$ into the pairs $ \set{\pm 1}, \set{\pm i}, \set{\pm j}, \set{\pm k} $.
\end{proof}

\subsection{States}
We now choose an initial state for $P^6$. Similarly to what was done in \cites{BattistaHyperbolic4manifolds,IMMAlgebraicFibering,IMMHyperbolic5Manifold} we choose a particularly symmetric kind of state. Note that, save for the facets $ A,B,C $, all the moves contain $6$ facets that are split in two disjoint triples of pairwise adjacent facets. In fact, one may check that in each of these moves, two facets are adjacent if and only if $ \lab $ assumes the same value on their label.

\begin{definition}
	A state is called \newterm{balanced} if it assigns the same status to $A$, $B$, and $C$, and if for every other move it assigns status $I$ to one of the two triples of pairwise adjacent facets, and $O$ to the facets of the other triple.
\end{definition}

In particular, a balanced state satisfies the property that, given two facets with label in $ \bintetra $ and in the same move, they are adjacent if and only if they have the same status. In particular, every balanced state is compatible.

\begin{lemma}
	If the initial state is balanced, then the orbit consists precisely of all balanced states.
\end{lemma}

\begin{proof}
	Since inverting the status of all facets in a move preserves the balancedness of a state, the orbit contains only balanced states. Let $s$ be the initial state, and $ s' $ be another balanced state. Moreover, let $ \mathcal F' $ be a set of facets containing exactly one facet for every move.

	By acting with the moves, we can send $s$ to another state $ s'' $ which coincides with $ s' $ on $ \mathcal F' $. Since a balanced state is uniquely determined by its restriction to $ \mathcal F' $, then $ s'' = s' $, and the action is therefore transitive on the balanced states.
\end{proof}

Therefore, we choose as initial state any balanced state, as the resulting orbit does not depend on the chosen balanced state.

By applying the Löbell construction on $P^6$, we obtain a $6$-manifold $ M^6 $, tessellated into copies of $P^6$, with an associated dual cube complex $\cubulation$ that is a spine of $M^6$. The initial state and moves produce a well-defined Bestvina-Brady Morse function $ f \colon \sd \cubulation \to S^1 $ as constructed in \cref{sec:bary-sub}. The rest of the section is devoted to proving the following.

\begin{proposition}\label{res:states-of-p6}
	Let $ P^6_v $ be a copy of $P^6$, with state $ \state[v] $, and let $ F $ be any face of $ P^6_v $ (possibly equal to $P^6_v$ itself). Then at least one of the following holds:
	\begin{itemize}
		\item the face $F$ is good;
		\item the inherited state is totally legal;
		\item the face $F$ is a vertex of $P^6_v$, and it is dual to a $6$-cube of $ \cube $ that decomposes as a product of three monochromatic squares.
	\end{itemize}
\end{proposition}

From this we get:

\begin{corollary}\label{res:f-regular-or-3-critical}
	Every ascending and descending link of $ f \colon \sd \cubulation \to S^1 $ is either collapsible, or collapses on a $2$-sphere.
\end{corollary}

\begin{proof}
	Let $ \cube $ be a cube of $\cubulation$, $v$ be a vertex of $\cube$, and let $F$ be the face of $P_v$ dual to $\cube$. By \cref{thm:inherited-state-collapsibility}, if the face $F$ satisfies one of the first two conditions, then the ascending and descending links at the barycentre of $\cube$ are collapsible.

	If the third condition holds instead, then $\cube$ is top-dimensional, so the link at the barycentre coincides with the face link, and by \cref{cor:face-link-good-cube-and-bad-squares} the ascending and descending face links collapse on a $2$-sphere.
\end{proof}

\subsection{Computing the states}

We start by analysing the states of the copies of $P^6$. To reduce the computations, we exploit the symmetries given by the group $\Rsixteen$.

\begin{lemma}\label{res:r16-preserves-balanced}
	The action of $ \Rsixteen $ on $P^6$ sends balanced states to balanced states.
\end{lemma}

\begin{proof}
	This follows from \cref{res:r16-preserves-moves} and the fact that the definition of balanced depends only on the set of moves as a partition.
\end{proof}

\begin{lemma}\label{res:r16-almost-transitive}
	Let $ \state, \varstate $ be two balanced states on $P^6$, that coincide on $A$, $B$, and $C$. Then there exists a $ \phi \in \Rsixteen $ that sends $ \state $ to $ \varstate$.
\end{lemma}

\begin{proof}
	Note that since $\Rsixteen$ fixes $A$, if we restrict the action to $A$ we obtain the same action on $P^5$ described in \cite{IMMHyperbolic5Manifold}.  Moreover, one can check that the restriction of the set of moves on $P^6$ coincides with the set of moves for $P^5$, so the restriction of a balanced state to the facets adjacent to $A$ yields a balanced state for $P^5$.

	By \cite[Proposition 9]{IMMHyperbolic5Manifold}, the action of $\Rsixteen$ on the balanced states of $P^5$ is transitive. Therefore, we may find a $ \phi \in \Rsixteen $ that sends $\state$ to a state that coincides with $ \varstate $ on all the facets adjacent to $A$, and by hypothesis also on $A$ (since $\Rsixteen$ does not change the status of $A$).

	The conclusion follows using that $ \phi(\state) $ is balanced by \cref{res:r16-preserves-balanced} and that a balanced state is uniquely determined by choosing the status of one facet in every move.
\end{proof}

In the following proofs, we make recurrent use of the following property of simplicial collapses: if $ v $ is a vertex of a simplicial complex $X$, and $ \link v $ is collapsible, then $X$ collapses on the subcomplex obtained by removing the open star of $v$.

\begin{proposition}\label{res:states-of-p6-are-legal}
	Let $\state$ be a balanced state on $P^6$. Then $ \state $ is totally legal.
\end{proposition}

\begin{proof}
	By \cref{res:r16-almost-transitive} it suffices to check two balanced states, one that assigns $ \Out $ to $A,B,C$ and one that assigns $ \In $ to the same facets. Therefore, we consider the state $\state$ that assigns status $\Out$ precisely to the facets
	\begin{center}
		\begin{tabular}{RRR}
			1, & 1-i+j-k,  & 1+i+j-k,  \\
			i, & 1+i+j+k,  & -1+i+j+k, \\
			j, & -1-i+j+k, & -1-i+j-k, \\
			k, & 1-i-j+k,  & 1-i+j+k,  \\
		\end{tabular}
	\end{center}
	and the state $ \varstate $ that assigns status $\Out$ to the facets above and to $A, B, C$. By symmetry, we only need to check the collapsibility of the full subcomplexes of $ \gosset $ whose vertices are dual to facets with status $\Out$; we denote them by $ \stateout, \varstateout $ respectively.

	Let us start by showing that $ \varstateout $ collapses to $ \stateout $. Consider the facet $A$, and look at its link inside $ \varstateout $. This is the full simplicial complex generated by the vertices on the second and third column, and it is collapsible since it is a cone on $ 1-i+j+k $. Similarly, the link of $B$ is generated by the vertices on the first and second column, and it is a cone on $ 1+i+j+k $, and the link of $C$, generated by vertices on the first and third column, is a cone on $j$. So we can collapse $A$, $B$, $C$ on their links in $ \varstateout $ and obtain $ \stateout $.

	Now we show that $ \stateout $ is collapsible.
	Consider $ -1-i+j-k $: its link has the vertices $ j $, $ 1-i+j-k $, $ -1-i+j+k $, $ 1+i+j-k $, $ -1+i+j+k $ and $ 1 -i+j+k $; it is a cone on $j$, so we can collapse $ -1-i+j-k $ on it. The remaining vertices all collapse on $ 1+i+j+k $. This concludes the proof.
\end{proof}

Next, we need to check the inherited states. By \cref{thm:inherited-state-collapsibility}, we only need to consider the inherited states of bad faces of $P^6$. For the following, fix a balanced state for $P^6$.

\begin{lemma}\label{res:bad-ridge-collapsible}
	Let $F_1, F_2$ be two adjacent facets in the same move, and let $ F=F_1 \cap F_2 $. Then the inherited state on $F$ is totally legal.
\end{lemma}

\begin{proof}
	Note that $F_1$ and $F_2$ must be on the same row of \cref{table:r-of-t}. Since $Q_8$ acts transitively on the rows, we may assume that $ \lab(F_1)=\lab(F_2)=i $. By using the involution of $ \Rsixteen $, we may also restrict to the case where $ F_1=1+i+j+k $ and $ F_2 = -1+i+j+k $, and the case where $ F_1 = i $ and $ F_2 = 1+i+j+k $.

	Let us start from the first case. The face $F$ is isometric to $P^4$, and its dual is the Gosset polytope $ 0_{21} $ shown in \cref{fig:P4-link}. Here vertices in the same move are depicted with the same colour.

	\newcommand{\Scale}{3}
	\newboolean{otherlink}
	\newboolean{showlinkone}
	\newboolean{showlinktwo}
	\newboolean{showlinkthree}
	\newboolean{showlinkfour}
	\newboolean{showlinkfive}
	\begin{figure}
		\centering
		\includegraphics[height=8cm]{pictures/P4.tex}
		\caption{The dual of the intersection of the facets with label $ \pm1+i+j+k $. Some lines are dotted only to make the picture cleaner.}
		\label{fig:P4-link}
	\end{figure}

	Since we started from a balanced state of $P^6$, the inherited state on $P^4$ has the property that two facets in the same move have the same status if and only if they are adjacent. Moreover, the facet with label $i$ has inherited status $ \Out $, since it belongs to the same move as $ \pm 1 +i+j+k $.

	Let us show the collapsibility of the subcomplex $X$ whose vertices have inherited status $\Out$. We proceed similarly to the proof of \cref{res:states-of-p6-are-legal}: we choose a vertex of $P^4$, and we show that if it has status $\Out$, then its link inside $X$ is collapsible, so we can collapse it on its link. Then we can remove the vertex from $P^4$ as it does not belong to $X$, and repeat the argument with a different vertex, until $X$ becomes a cone.

	\renewcommand{\Scale}{2}
	\begin{figure}
		\centering
		\setboolean{showlinkone}{true}
		\begin{subfigure}{0.45\textwidth}
			\begin{tikzpicture}[scale=\Scale]
	\usetikzlibrary{positioning}
	\usetikzlibrary{calc}
	\usetikzlibrary{backgrounds,fit}
	\usetikzlibrary{patterns}

	\newcommand\Star[4][]{%
		\path[#1] #4 +(90  :#3) -- +( 126:#2)
		-- +(162 :#3) -- +(198:#2)
		-- +(234:#3) -- +(270:#2)
		-- +(306:#3) -- +(342:#2)
		-- +(18:#3) -- +(54:#2) -- +(90: #3);}

	\tikzset{
	-|>-/.style={
	decoration={
			markings,
			mark=at position #1 with {\arrow{latex}}
		},
	postaction={decorate}},
	-|>-/.default=.5
	}

	\tikzset{
	->-/.style={
	decoration={
			markings,
			mark=at position #1 with {\arrow{>}}
		},
	postaction={decorate}},
	->-/.default=.53
	}

	\tikzset{
		every node/.style={
				fill=white,
				text opacity=1,
				fill opacity=.9,
				rounded corners=10pt,
			},
	}

	\def\r{1}
	\def\R{1.6}
	\ifotherlink
		\useasboundingbox (-2.4,-1.5) rectangle (2.4,1.9);
	\fi
	\coordinate (A) at (90: \R);
	\coordinate (B) at (162: \R);
	\coordinate (C) at (234: \R);
	\coordinate (D) at (306: \R);
	\coordinate (E) at (18: \R);
	\coordinate (A1) at (270: \r);
	\coordinate (B1) at (342: \r);
	\coordinate (C1) at (414: \r);
	\coordinate (D1) at (126: \r);
	\coordinate (E1) at (198: \r);

	\ifshowlinkone
		\draw
		(B) -- (D1) -- (C) -- cycle
		(E) -- (C1) -- (D) -- cycle
		(C) -- (D)
		(D1) -- (E)
		($(B) !.43 ! (C1)$) -- (C1);
		;
		\draw[dotted]
		($(B) !.43 ! (C1)$) -- (B);
	\else \ifshowlinktwo
			\draw
			(E) -- (D1) -- (B1) -- cycle
			(E1) -- (B)
			(E1) -- (B1)
			(D1) -- (B)
			(E) -- (C1)
			($(E1) !.87 ! (C1)$) -- (C1)
			($(E1) !.62 ! (C1)$) -- (E1)
			($(B) !.73 ! (C1)$) -- (C1)
			($(B) !.53 ! (C1)$) -- (B);
			;
			\draw[dotted]
			($(B) !.73 ! (C1)$) -- ($(B) !.53 ! (C1)$)
			($(E1) !.62 ! (C1)$) -- ($(E1) !.87 ! (C1)$)
			;
		\else \ifshowlinkthree
				\draw (B) -- (C) -- (B1) -- (E);
			\else\ifshowlinkfour
					\draw (E) -- (D) -- (E1) -- (B);
				\else\ifshowlinkfive
						\draw
						(B) -- (C) -- (E1) -- cycle
						(D) -- (E) -- (B1) -- cycle
						(B1) -- (E1) -- (D) -- (C)
						;
						\draw[dotted] (B1) -- (C);
					\else
						\draw
						(A) -- (B) -- (C) -- (D) -- (E) -- cycle
						(A1) -- (C1) -- (E1) -- (B1) -- (D1) -- cycle
						(A1) -- (C) -- (E1) -- (B) -- (D1) -- (A) -- (C1) -- (E) -- (B1) -- (D) -- cycle
						;
						\draw[dotted]
						(A1) -- (B) -- (C1) -- (D) -- (E1) -- (A) -- (B1) -- (C) -- (D1) -- (E) -- cycle
						;
					\fi \fi \fi \fi \fi

	\def\r{0.04}
	\draw
	\ifshowlinkone
	(D) node[below right]{$-1-i+j+k$}
	(C) node[below left]{$1-i+j+k$}
	(B) node[below=.2cm]{$1+i+j-k$}
	(E) node[below=.2cm]{$-1+i-j+k$}
	(A1) node[above=.2cm]{$A$}
	(D1) node[above left]{$1+i-j+k$}
	(C1) node[above right]{$-1+i+j-k$}
	\else\ifshowlinktwo
	(B1) node[below right]{$k$}
	(E) node[below=.2cm]{$-1+i-j+k$}
	(E1) node[below left]{$j$}
	(B) node[below=.2cm]{$1+i+j-k$}
	(A) node[above=.3cm]{$i$}
	(D1) node[above left]{$1+i-j+k$}
	(C1) node[above right]{$-1+i+j-k$}
	\else\ifshowlinkthree
	(B1) node[below right]{$k$}
	(C) node[below=.2cm]{$1-i+j+k$}
	(B) node[right=.1cm]{$1+i+j-k$}
	(E) node[left=.1cm]{$-1+i-j+k$}
	(D1) node[above=.2cm]{$1+i-j+k$}
	\else
	\ifshowlinkfour
	(D) node[below=.2cm]{$-1-i+j+k$}
	(E) node[left=.1cm]{$-1+i-j+k$}
	(B) node[right=.1cm]{$1+i+j-k$}
	(E1) node[below left]{$j$}
	(C1) node[above=.2cm]{$-1+i+j-k$}
	\else\ifshowlinkfive
	(B) node[above right]{$1+i+j-k$}
	(C) node[below=.2cm]{$1-i+j+k$}
	(D) node[below=.2cm]{$-1-i+j+k$}
	(E) node[above left]{$-1+i-j+k$}
	(B1) node[above left]{$k$}
	(E1) node[above right]{$j$}
	\else\ifotherlink
	(A) node[above=.3cm]{$-1+i+j+k$}
	(B) node[above left=.1cm]{$k$}
	(C) node[below left=.1cm]{$1$}
	(D) node[below right]{$1+i-j-k$}
	(E) node[above right]{$-1+i+j-k$}
	(A1) node[below=.2cm]{$B$}
	(B1) node[below right]{$1+i+j-k$}
	(C1) node[above right]{$-1+i-j+k$}
	(D1) node[above left]{$j$}
	(E1) node[below left]{$1+i-j+k$}
	\else
	(A) node[above=.3cm]{$i$}
	(B) node[above left]{$1+i+j-k$}
	(C) node[below left]{$1-i+j+k$}
	(D) node[below right]{$-1-i+j+k$}
	(E) node[above right]{$-1+i-j+k$}
	(A1) node[below=.2cm]{$A$}
	(B1) node[below right=.1cm]{$k$}
	(C1) node[above right]{$-1+i+j-k$}
	(D1) node[above left]{$1+i-j+k$}
	(E1) node[below left=.1cm]{$j$}
	\fi\fi\fi\fi\fi\fi

	;

	\unless\ifshowlinkfive
		\ifshowlinkfour
			\Star[fill=blue,draw]{\r*0.7}{\r*1.4}{(C1)}
		\else
			\draw[fill=blue] (C1) circle (\r);
			\ifshowlinkthree
				\Star[fill=green,draw]{\r*0.7}{\r*1.4}{(D1)}
			\else
				\draw[fill=green] (D1) circle (\r);
				\ifshowlinktwo
					\Star[fill=white,draw]{\r*0.7}{\r*1.4}{(A)}
				\else
					\draw[fill=white] (A) circle (\r);
					\ifshowlinkone
						\Star[fill=black,draw]{\r*0.7}{\r*1.4}{(A1)}
					\else
						\draw[fill=black] (A1) circle (\r);
					\fi
				\fi
			\fi
		\fi
	\fi
	\draw[fill=blue] (B1) circle (\r);
	\draw[fill=green] (E1) circle (\r);
	\draw[fill=red] (B) circle (\r);
	\draw[fill=blue] (C) circle (\r);
	\draw[fill=red] (E) circle (\r);
	\draw[fill=green] (D) circle (\r);

\end{tikzpicture}
		\end{subfigure}\hspace{.5cm}
		\setboolean{showlinkone}{false}
		\setboolean{showlinktwo}{true}
		\begin{subfigure}{0.45\textwidth}
			\begin{tikzpicture}[scale=\Scale]
	\usetikzlibrary{positioning}
	\usetikzlibrary{calc}
	\usetikzlibrary{backgrounds,fit}
	\usetikzlibrary{patterns}

	\newcommand\Star[4][]{%
		\path[#1] #4 +(90  :#3) -- +( 126:#2)
		-- +(162 :#3) -- +(198:#2)
		-- +(234:#3) -- +(270:#2)
		-- +(306:#3) -- +(342:#2)
		-- +(18:#3) -- +(54:#2) -- +(90: #3);}

	\tikzset{
	-|>-/.style={
	decoration={
			markings,
			mark=at position #1 with {\arrow{latex}}
		},
	postaction={decorate}},
	-|>-/.default=.5
	}

	\tikzset{
	->-/.style={
	decoration={
			markings,
			mark=at position #1 with {\arrow{>}}
		},
	postaction={decorate}},
	->-/.default=.53
	}

	\tikzset{
		every node/.style={
				fill=white,
				text opacity=1,
				fill opacity=.9,
				rounded corners=10pt,
			},
	}

	\def\r{1}
	\def\R{1.6}
	\ifotherlink
		\useasboundingbox (-2.4,-1.5) rectangle (2.4,1.9);
	\fi
	\coordinate (A) at (90: \R);
	\coordinate (B) at (162: \R);
	\coordinate (C) at (234: \R);
	\coordinate (D) at (306: \R);
	\coordinate (E) at (18: \R);
	\coordinate (A1) at (270: \r);
	\coordinate (B1) at (342: \r);
	\coordinate (C1) at (414: \r);
	\coordinate (D1) at (126: \r);
	\coordinate (E1) at (198: \r);

	\ifshowlinkone
		\draw
		(B) -- (D1) -- (C) -- cycle
		(E) -- (C1) -- (D) -- cycle
		(C) -- (D)
		(D1) -- (E)
		($(B) !.43 ! (C1)$) -- (C1);
		;
		\draw[dotted]
		($(B) !.43 ! (C1)$) -- (B);
	\else \ifshowlinktwo
			\draw
			(E) -- (D1) -- (B1) -- cycle
			(E1) -- (B)
			(E1) -- (B1)
			(D1) -- (B)
			(E) -- (C1)
			($(E1) !.87 ! (C1)$) -- (C1)
			($(E1) !.62 ! (C1)$) -- (E1)
			($(B) !.73 ! (C1)$) -- (C1)
			($(B) !.53 ! (C1)$) -- (B);
			;
			\draw[dotted]
			($(B) !.73 ! (C1)$) -- ($(B) !.53 ! (C1)$)
			($(E1) !.62 ! (C1)$) -- ($(E1) !.87 ! (C1)$)
			;
		\else \ifshowlinkthree
				\draw (B) -- (C) -- (B1) -- (E);
			\else\ifshowlinkfour
					\draw (E) -- (D) -- (E1) -- (B);
				\else\ifshowlinkfive
						\draw
						(B) -- (C) -- (E1) -- cycle
						(D) -- (E) -- (B1) -- cycle
						(B1) -- (E1) -- (D) -- (C)
						;
						\draw[dotted] (B1) -- (C);
					\else
						\draw
						(A) -- (B) -- (C) -- (D) -- (E) -- cycle
						(A1) -- (C1) -- (E1) -- (B1) -- (D1) -- cycle
						(A1) -- (C) -- (E1) -- (B) -- (D1) -- (A) -- (C1) -- (E) -- (B1) -- (D) -- cycle
						;
						\draw[dotted]
						(A1) -- (B) -- (C1) -- (D) -- (E1) -- (A) -- (B1) -- (C) -- (D1) -- (E) -- cycle
						;
					\fi \fi \fi \fi \fi

	\def\r{0.04}
	\draw
	\ifshowlinkone
	(D) node[below right]{$-1-i+j+k$}
	(C) node[below left]{$1-i+j+k$}
	(B) node[below=.2cm]{$1+i+j-k$}
	(E) node[below=.2cm]{$-1+i-j+k$}
	(A1) node[above=.2cm]{$A$}
	(D1) node[above left]{$1+i-j+k$}
	(C1) node[above right]{$-1+i+j-k$}
	\else\ifshowlinktwo
	(B1) node[below right]{$k$}
	(E) node[below=.2cm]{$-1+i-j+k$}
	(E1) node[below left]{$j$}
	(B) node[below=.2cm]{$1+i+j-k$}
	(A) node[above=.3cm]{$i$}
	(D1) node[above left]{$1+i-j+k$}
	(C1) node[above right]{$-1+i+j-k$}
	\else\ifshowlinkthree
	(B1) node[below right]{$k$}
	(C) node[below=.2cm]{$1-i+j+k$}
	(B) node[right=.1cm]{$1+i+j-k$}
	(E) node[left=.1cm]{$-1+i-j+k$}
	(D1) node[above=.2cm]{$1+i-j+k$}
	\else
	\ifshowlinkfour
	(D) node[below=.2cm]{$-1-i+j+k$}
	(E) node[left=.1cm]{$-1+i-j+k$}
	(B) node[right=.1cm]{$1+i+j-k$}
	(E1) node[below left]{$j$}
	(C1) node[above=.2cm]{$-1+i+j-k$}
	\else\ifshowlinkfive
	(B) node[above right]{$1+i+j-k$}
	(C) node[below=.2cm]{$1-i+j+k$}
	(D) node[below=.2cm]{$-1-i+j+k$}
	(E) node[above left]{$-1+i-j+k$}
	(B1) node[above left]{$k$}
	(E1) node[above right]{$j$}
	\else\ifotherlink
	(A) node[above=.3cm]{$-1+i+j+k$}
	(B) node[above left=.1cm]{$k$}
	(C) node[below left=.1cm]{$1$}
	(D) node[below right]{$1+i-j-k$}
	(E) node[above right]{$-1+i+j-k$}
	(A1) node[below=.2cm]{$B$}
	(B1) node[below right]{$1+i+j-k$}
	(C1) node[above right]{$-1+i-j+k$}
	(D1) node[above left]{$j$}
	(E1) node[below left]{$1+i-j+k$}
	\else
	(A) node[above=.3cm]{$i$}
	(B) node[above left]{$1+i+j-k$}
	(C) node[below left]{$1-i+j+k$}
	(D) node[below right]{$-1-i+j+k$}
	(E) node[above right]{$-1+i-j+k$}
	(A1) node[below=.2cm]{$A$}
	(B1) node[below right=.1cm]{$k$}
	(C1) node[above right]{$-1+i+j-k$}
	(D1) node[above left]{$1+i-j+k$}
	(E1) node[below left=.1cm]{$j$}
	\fi\fi\fi\fi\fi\fi

	;

	\unless\ifshowlinkfive
		\ifshowlinkfour
			\Star[fill=blue,draw]{\r*0.7}{\r*1.4}{(C1)}
		\else
			\draw[fill=blue] (C1) circle (\r);
			\ifshowlinkthree
				\Star[fill=green,draw]{\r*0.7}{\r*1.4}{(D1)}
			\else
				\draw[fill=green] (D1) circle (\r);
				\ifshowlinktwo
					\Star[fill=white,draw]{\r*0.7}{\r*1.4}{(A)}
				\else
					\draw[fill=white] (A) circle (\r);
					\ifshowlinkone
						\Star[fill=black,draw]{\r*0.7}{\r*1.4}{(A1)}
					\else
						\draw[fill=black] (A1) circle (\r);
					\fi
				\fi
			\fi
		\fi
	\fi
	\draw[fill=blue] (B1) circle (\r);
	\draw[fill=green] (E1) circle (\r);
	\draw[fill=red] (B) circle (\r);
	\draw[fill=blue] (C) circle (\r);
	\draw[fill=red] (E) circle (\r);
	\draw[fill=green] (D) circle (\r);

\end{tikzpicture}
		\end{subfigure}
		\caption{The links of $A$ and $i$ inside $P^4$: they both form a triangular prism.}
		\label{fig:P4-links-A-i}
	\end{figure}

	Consider the vertex $A$, and suppose it has inherited status $\Out$. Its link inside $P^4$ is the one depicted in \cref{fig:P4-links-A-i}, left. This is a triangular prism dual to $P^3$, with the usual set of moves considered in \cite{BattistaHyperbolic4manifolds}. The balancedness property implies that the two vertices in the same move have opposite status, so the subcomplex of the prism generated by vertices with status $\Out$ is collapsible.
	Therefore, the link of $A$ in $X$ is collapsible.

	Next, we consider the vertex with label $i$, which has inherited status $\Out$; its link is in \cref{fig:P4-links-A-i}, right. By the same argument as above, its link in $X$ is collapsible.

	\renewcommand{\Scale}{2.5}
	\begin{figure}
		\centering
		\setboolean{showlinktwo}{false}
		\setboolean{showlinkthree}{true}
		\begin{subfigure}{0.4\textwidth}
			\begin{tikzpicture}[scale=\Scale]
	\usetikzlibrary{positioning}
	\usetikzlibrary{calc}
	\usetikzlibrary{backgrounds,fit}
	\usetikzlibrary{patterns}

	\newcommand\Star[4][]{%
		\path[#1] #4 +(90  :#3) -- +( 126:#2)
		-- +(162 :#3) -- +(198:#2)
		-- +(234:#3) -- +(270:#2)
		-- +(306:#3) -- +(342:#2)
		-- +(18:#3) -- +(54:#2) -- +(90: #3);}

	\tikzset{
	-|>-/.style={
	decoration={
			markings,
			mark=at position #1 with {\arrow{latex}}
		},
	postaction={decorate}},
	-|>-/.default=.5
	}

	\tikzset{
	->-/.style={
	decoration={
			markings,
			mark=at position #1 with {\arrow{>}}
		},
	postaction={decorate}},
	->-/.default=.53
	}

	\tikzset{
		every node/.style={
				fill=white,
				text opacity=1,
				fill opacity=.9,
				rounded corners=10pt,
			},
	}

	\def\r{1}
	\def\R{1.6}
	\ifotherlink
		\useasboundingbox (-2.4,-1.5) rectangle (2.4,1.9);
	\fi
	\coordinate (A) at (90: \R);
	\coordinate (B) at (162: \R);
	\coordinate (C) at (234: \R);
	\coordinate (D) at (306: \R);
	\coordinate (E) at (18: \R);
	\coordinate (A1) at (270: \r);
	\coordinate (B1) at (342: \r);
	\coordinate (C1) at (414: \r);
	\coordinate (D1) at (126: \r);
	\coordinate (E1) at (198: \r);

	\ifshowlinkone
		\draw
		(B) -- (D1) -- (C) -- cycle
		(E) -- (C1) -- (D) -- cycle
		(C) -- (D)
		(D1) -- (E)
		($(B) !.43 ! (C1)$) -- (C1);
		;
		\draw[dotted]
		($(B) !.43 ! (C1)$) -- (B);
	\else \ifshowlinktwo
			\draw
			(E) -- (D1) -- (B1) -- cycle
			(E1) -- (B)
			(E1) -- (B1)
			(D1) -- (B)
			(E) -- (C1)
			($(E1) !.87 ! (C1)$) -- (C1)
			($(E1) !.62 ! (C1)$) -- (E1)
			($(B) !.73 ! (C1)$) -- (C1)
			($(B) !.53 ! (C1)$) -- (B);
			;
			\draw[dotted]
			($(B) !.73 ! (C1)$) -- ($(B) !.53 ! (C1)$)
			($(E1) !.62 ! (C1)$) -- ($(E1) !.87 ! (C1)$)
			;
		\else \ifshowlinkthree
				\draw (B) -- (C) -- (B1) -- (E);
			\else\ifshowlinkfour
					\draw (E) -- (D) -- (E1) -- (B);
				\else\ifshowlinkfive
						\draw
						(B) -- (C) -- (E1) -- cycle
						(D) -- (E) -- (B1) -- cycle
						(B1) -- (E1) -- (D) -- (C)
						;
						\draw[dotted] (B1) -- (C);
					\else
						\draw
						(A) -- (B) -- (C) -- (D) -- (E) -- cycle
						(A1) -- (C1) -- (E1) -- (B1) -- (D1) -- cycle
						(A1) -- (C) -- (E1) -- (B) -- (D1) -- (A) -- (C1) -- (E) -- (B1) -- (D) -- cycle
						;
						\draw[dotted]
						(A1) -- (B) -- (C1) -- (D) -- (E1) -- (A) -- (B1) -- (C) -- (D1) -- (E) -- cycle
						;
					\fi \fi \fi \fi \fi

	\def\r{0.04}
	\draw
	\ifshowlinkone
	(D) node[below right]{$-1-i+j+k$}
	(C) node[below left]{$1-i+j+k$}
	(B) node[below=.2cm]{$1+i+j-k$}
	(E) node[below=.2cm]{$-1+i-j+k$}
	(A1) node[above=.2cm]{$A$}
	(D1) node[above left]{$1+i-j+k$}
	(C1) node[above right]{$-1+i+j-k$}
	\else\ifshowlinktwo
	(B1) node[below right]{$k$}
	(E) node[below=.2cm]{$-1+i-j+k$}
	(E1) node[below left]{$j$}
	(B) node[below=.2cm]{$1+i+j-k$}
	(A) node[above=.3cm]{$i$}
	(D1) node[above left]{$1+i-j+k$}
	(C1) node[above right]{$-1+i+j-k$}
	\else\ifshowlinkthree
	(B1) node[below right]{$k$}
	(C) node[below=.2cm]{$1-i+j+k$}
	(B) node[right=.1cm]{$1+i+j-k$}
	(E) node[left=.1cm]{$-1+i-j+k$}
	(D1) node[above=.2cm]{$1+i-j+k$}
	\else
	\ifshowlinkfour
	(D) node[below=.2cm]{$-1-i+j+k$}
	(E) node[left=.1cm]{$-1+i-j+k$}
	(B) node[right=.1cm]{$1+i+j-k$}
	(E1) node[below left]{$j$}
	(C1) node[above=.2cm]{$-1+i+j-k$}
	\else\ifshowlinkfive
	(B) node[above right]{$1+i+j-k$}
	(C) node[below=.2cm]{$1-i+j+k$}
	(D) node[below=.2cm]{$-1-i+j+k$}
	(E) node[above left]{$-1+i-j+k$}
	(B1) node[above left]{$k$}
	(E1) node[above right]{$j$}
	\else\ifotherlink
	(A) node[above=.3cm]{$-1+i+j+k$}
	(B) node[above left=.1cm]{$k$}
	(C) node[below left=.1cm]{$1$}
	(D) node[below right]{$1+i-j-k$}
	(E) node[above right]{$-1+i+j-k$}
	(A1) node[below=.2cm]{$B$}
	(B1) node[below right]{$1+i+j-k$}
	(C1) node[above right]{$-1+i-j+k$}
	(D1) node[above left]{$j$}
	(E1) node[below left]{$1+i-j+k$}
	\else
	(A) node[above=.3cm]{$i$}
	(B) node[above left]{$1+i+j-k$}
	(C) node[below left]{$1-i+j+k$}
	(D) node[below right]{$-1-i+j+k$}
	(E) node[above right]{$-1+i-j+k$}
	(A1) node[below=.2cm]{$A$}
	(B1) node[below right=.1cm]{$k$}
	(C1) node[above right]{$-1+i+j-k$}
	(D1) node[above left]{$1+i-j+k$}
	(E1) node[below left=.1cm]{$j$}
	\fi\fi\fi\fi\fi\fi

	;

	\unless\ifshowlinkfive
		\ifshowlinkfour
			\Star[fill=blue,draw]{\r*0.7}{\r*1.4}{(C1)}
		\else
			\draw[fill=blue] (C1) circle (\r);
			\ifshowlinkthree
				\Star[fill=green,draw]{\r*0.7}{\r*1.4}{(D1)}
			\else
				\draw[fill=green] (D1) circle (\r);
				\ifshowlinktwo
					\Star[fill=white,draw]{\r*0.7}{\r*1.4}{(A)}
				\else
					\draw[fill=white] (A) circle (\r);
					\ifshowlinkone
						\Star[fill=black,draw]{\r*0.7}{\r*1.4}{(A1)}
					\else
						\draw[fill=black] (A1) circle (\r);
					\fi
				\fi
			\fi
		\fi
	\fi
	\draw[fill=blue] (B1) circle (\r);
	\draw[fill=green] (E1) circle (\r);
	\draw[fill=red] (B) circle (\r);
	\draw[fill=blue] (C) circle (\r);
	\draw[fill=red] (E) circle (\r);
	\draw[fill=green] (D) circle (\r);

\end{tikzpicture}
		\end{subfigure}\hspace{.5cm}
		\setboolean{showlinkthree}{false}
		\setboolean{showlinkfour}{true}
		\begin{subfigure}{0.4\textwidth}
			\begin{tikzpicture}[scale=\Scale]
	\usetikzlibrary{positioning}
	\usetikzlibrary{calc}
	\usetikzlibrary{backgrounds,fit}
	\usetikzlibrary{patterns}

	\newcommand\Star[4][]{%
		\path[#1] #4 +(90  :#3) -- +( 126:#2)
		-- +(162 :#3) -- +(198:#2)
		-- +(234:#3) -- +(270:#2)
		-- +(306:#3) -- +(342:#2)
		-- +(18:#3) -- +(54:#2) -- +(90: #3);}

	\tikzset{
	-|>-/.style={
	decoration={
			markings,
			mark=at position #1 with {\arrow{latex}}
		},
	postaction={decorate}},
	-|>-/.default=.5
	}

	\tikzset{
	->-/.style={
	decoration={
			markings,
			mark=at position #1 with {\arrow{>}}
		},
	postaction={decorate}},
	->-/.default=.53
	}

	\tikzset{
		every node/.style={
				fill=white,
				text opacity=1,
				fill opacity=.9,
				rounded corners=10pt,
			},
	}

	\def\r{1}
	\def\R{1.6}
	\ifotherlink
		\useasboundingbox (-2.4,-1.5) rectangle (2.4,1.9);
	\fi
	\coordinate (A) at (90: \R);
	\coordinate (B) at (162: \R);
	\coordinate (C) at (234: \R);
	\coordinate (D) at (306: \R);
	\coordinate (E) at (18: \R);
	\coordinate (A1) at (270: \r);
	\coordinate (B1) at (342: \r);
	\coordinate (C1) at (414: \r);
	\coordinate (D1) at (126: \r);
	\coordinate (E1) at (198: \r);

	\ifshowlinkone
		\draw
		(B) -- (D1) -- (C) -- cycle
		(E) -- (C1) -- (D) -- cycle
		(C) -- (D)
		(D1) -- (E)
		($(B) !.43 ! (C1)$) -- (C1);
		;
		\draw[dotted]
		($(B) !.43 ! (C1)$) -- (B);
	\else \ifshowlinktwo
			\draw
			(E) -- (D1) -- (B1) -- cycle
			(E1) -- (B)
			(E1) -- (B1)
			(D1) -- (B)
			(E) -- (C1)
			($(E1) !.87 ! (C1)$) -- (C1)
			($(E1) !.62 ! (C1)$) -- (E1)
			($(B) !.73 ! (C1)$) -- (C1)
			($(B) !.53 ! (C1)$) -- (B);
			;
			\draw[dotted]
			($(B) !.73 ! (C1)$) -- ($(B) !.53 ! (C1)$)
			($(E1) !.62 ! (C1)$) -- ($(E1) !.87 ! (C1)$)
			;
		\else \ifshowlinkthree
				\draw (B) -- (C) -- (B1) -- (E);
			\else\ifshowlinkfour
					\draw (E) -- (D) -- (E1) -- (B);
				\else\ifshowlinkfive
						\draw
						(B) -- (C) -- (E1) -- cycle
						(D) -- (E) -- (B1) -- cycle
						(B1) -- (E1) -- (D) -- (C)
						;
						\draw[dotted] (B1) -- (C);
					\else
						\draw
						(A) -- (B) -- (C) -- (D) -- (E) -- cycle
						(A1) -- (C1) -- (E1) -- (B1) -- (D1) -- cycle
						(A1) -- (C) -- (E1) -- (B) -- (D1) -- (A) -- (C1) -- (E) -- (B1) -- (D) -- cycle
						;
						\draw[dotted]
						(A1) -- (B) -- (C1) -- (D) -- (E1) -- (A) -- (B1) -- (C) -- (D1) -- (E) -- cycle
						;
					\fi \fi \fi \fi \fi

	\def\r{0.04}
	\draw
	\ifshowlinkone
	(D) node[below right]{$-1-i+j+k$}
	(C) node[below left]{$1-i+j+k$}
	(B) node[below=.2cm]{$1+i+j-k$}
	(E) node[below=.2cm]{$-1+i-j+k$}
	(A1) node[above=.2cm]{$A$}
	(D1) node[above left]{$1+i-j+k$}
	(C1) node[above right]{$-1+i+j-k$}
	\else\ifshowlinktwo
	(B1) node[below right]{$k$}
	(E) node[below=.2cm]{$-1+i-j+k$}
	(E1) node[below left]{$j$}
	(B) node[below=.2cm]{$1+i+j-k$}
	(A) node[above=.3cm]{$i$}
	(D1) node[above left]{$1+i-j+k$}
	(C1) node[above right]{$-1+i+j-k$}
	\else\ifshowlinkthree
	(B1) node[below right]{$k$}
	(C) node[below=.2cm]{$1-i+j+k$}
	(B) node[right=.1cm]{$1+i+j-k$}
	(E) node[left=.1cm]{$-1+i-j+k$}
	(D1) node[above=.2cm]{$1+i-j+k$}
	\else
	\ifshowlinkfour
	(D) node[below=.2cm]{$-1-i+j+k$}
	(E) node[left=.1cm]{$-1+i-j+k$}
	(B) node[right=.1cm]{$1+i+j-k$}
	(E1) node[below left]{$j$}
	(C1) node[above=.2cm]{$-1+i+j-k$}
	\else\ifshowlinkfive
	(B) node[above right]{$1+i+j-k$}
	(C) node[below=.2cm]{$1-i+j+k$}
	(D) node[below=.2cm]{$-1-i+j+k$}
	(E) node[above left]{$-1+i-j+k$}
	(B1) node[above left]{$k$}
	(E1) node[above right]{$j$}
	\else\ifotherlink
	(A) node[above=.3cm]{$-1+i+j+k$}
	(B) node[above left=.1cm]{$k$}
	(C) node[below left=.1cm]{$1$}
	(D) node[below right]{$1+i-j-k$}
	(E) node[above right]{$-1+i+j-k$}
	(A1) node[below=.2cm]{$B$}
	(B1) node[below right]{$1+i+j-k$}
	(C1) node[above right]{$-1+i-j+k$}
	(D1) node[above left]{$j$}
	(E1) node[below left]{$1+i-j+k$}
	\else
	(A) node[above=.3cm]{$i$}
	(B) node[above left]{$1+i+j-k$}
	(C) node[below left]{$1-i+j+k$}
	(D) node[below right]{$-1-i+j+k$}
	(E) node[above right]{$-1+i-j+k$}
	(A1) node[below=.2cm]{$A$}
	(B1) node[below right=.1cm]{$k$}
	(C1) node[above right]{$-1+i+j-k$}
	(D1) node[above left]{$1+i-j+k$}
	(E1) node[below left=.1cm]{$j$}
	\fi\fi\fi\fi\fi\fi

	;

	\unless\ifshowlinkfive
		\ifshowlinkfour
			\Star[fill=blue,draw]{\r*0.7}{\r*1.4}{(C1)}
		\else
			\draw[fill=blue] (C1) circle (\r);
			\ifshowlinkthree
				\Star[fill=green,draw]{\r*0.7}{\r*1.4}{(D1)}
			\else
				\draw[fill=green] (D1) circle (\r);
				\ifshowlinktwo
					\Star[fill=white,draw]{\r*0.7}{\r*1.4}{(A)}
				\else
					\draw[fill=white] (A) circle (\r);
					\ifshowlinkone
						\Star[fill=black,draw]{\r*0.7}{\r*1.4}{(A1)}
					\else
						\draw[fill=black] (A1) circle (\r);
					\fi
				\fi
			\fi
		\fi
	\fi
	\draw[fill=blue] (B1) circle (\r);
	\draw[fill=green] (E1) circle (\r);
	\draw[fill=red] (B) circle (\r);
	\draw[fill=blue] (C) circle (\r);
	\draw[fill=red] (E) circle (\r);
	\draw[fill=green] (D) circle (\r);

\end{tikzpicture}
		\end{subfigure}
		\caption{The links of $1+i-j+k$ and $-1+i+j-k$ inside $P^4$ with the vertices $A$ and $i$ removed.}
		\label{fig:P4-links-three-four}
	\end{figure}

	Then, we consider the vertex $ 1+i-j+k $. Its link is the path shown in \cref{fig:P4-links-three-four}, left. By the balancedness property, $ 1-i+j+k $ and $k$ have the same status, while $ 1+i+j-k $ and $ -1+i-j+k $ have opposite status. By inspecting the four possible cases, the subcomplex of the link whose vertices have status $\Out$ is always collapsible.
	The same applies for the vertex $ -1+i+j-k $, see \cref{fig:P4-links-three-four}, right.

	After removing $ A $, $i$, $ 1+i-j+k $ and $ -1+i+j-k $ we are left with the simplicial complex shown in \cref{fig:P4-link-five}.
	\renewcommand{\Scale}{3}
	\setboolean{showlinkfour}{false}
	\setboolean{showlinkfive}{true}
	\begin{figure}
		\centering
		\includegraphics[width=0.5\textwidth]{pictures/P4.tex}
		\caption{The simplicial complex obtained from $ P^4 $ by removing four vertices.}
		\label{fig:P4-link-five}
	\end{figure}

	Again, by exploiting the balancedness of the state, we know that $j$ has the same status as $ -1-i+j+k $, $k$ has the same status of $ 1-i+j+k $, while $ 1+i+j-k $ and $ -1+i-j+k $ have opposite status. By checking the eight possible configurations, it is easy to check that the subcomplex generated by vertices with status $\Out$ is either a point, a path of length two, or the union of a tetrahedron and a triangle, joined by a common edge. All of them are collapsible, so this concludes the proof that $X$ is collapsible.

	By using the same arguments one can also show the collapsibility of the subcomplex of vertices with status $\In$. The remaining case where $F_1 = i$ and $F_2 = 1+i+j+k$ is also analogous: indeed, the dual of their intersection is the $P^4$ in \cref{fig:other-P4}, and by comparing with \cref{fig:P4-link} it is apparent that it is combinatorially isomorphic to the first case.
	\setboolean{showlinkfive}{false}
	\setboolean{otherlink}{true}
	\begin{figure}
		\centering
		\includegraphics[height=8cm]{pictures/P4.tex}
		\caption{The dual of the intersection of the facets $ i $ and $ 1+i+j+k $.}
		\label{fig:other-P4}
	\end{figure}
\end{proof}

Unlike what happened for $P^5$, here we also have bad faces of codimension higher than $2$. We classify them in the following lemma.

\begin{lemma}
	Let $F$ be a proper bad face of $P^6$. Then $F$ is one of the following:
	\begin{enumerate}
		\item the intersection of two facets in the same move;\label{i:bad-square}
		\item the intersection of three facets in the same move;\label{i:bad-cube}
		\item the intersection of four facets belonging to two moves (two facets per move);\label{i:bad-squarexsquare}
		\item the intersection of six facets belonging to three moves (two facets per move).\label{i:bad-squarexsquarexsquare}
	\end{enumerate}
\end{lemma}

\begin{proof}
	Bad faces of codimension $k$ of $P^6$ are obtained by choosing a subset of $k$ pairwise adjacent facets so that no move contains exactly one facet from this subset. Clearly, there must be at least a move containing at least two facets of this subset: using the action of $ \Rsixteen $, we may assume that these are $ 1+i+j+k $ and $ -1+i+j+k $, or $ i $ and $ 1+i+j+k $.

	Let us consider the first case: the other $k-2$ facets must form a clique in the graph in \cref{fig:P4-link}.

	\begin{itemize}
		\item If $k=2$, there is only one possibility, and we are in case \cref{i:bad-square}.
		\item If $k=3$, then the only possibility is to choose $i$, as in \cref{i:bad-cube}.
		\item If $ k=4 $, to get a bad face we must choose two adjacent vertices in \cref{fig:P4-link} that are in the same move. There are two ways to do that, by choosing $ j $ and $ -1-i++j+k $, or $ k $ and $ 1-i+j+j $. In both cases the resulting face is as in \cref{i:bad-squarexsquare}.
		\item If $k=5$, we should either choose $i$ and other two vertices in the same move adjacent to it, or three pairwise adjacent vertices all belonging to the same move. However, there are no ways to do this, so there are no codimension $5$ bad faces.
		\item If $k=6$, we should either choose $i$ and a monochromatic triangle, or two monochromatic segments that span a tetrahedron. While there is no monochromatic triangle in \cref{fig:P4-link}, the two pairs $ j $, $ -1-i+j+k $ and $ k $, $ 1-i+j+k $ do satisfy the hypotheses. By intersecting them with $ \pm 1 +i+j+k $, we obtain a codimension $6$ face, i.e.~a vertex, as in \cref{i:bad-squarexsquarexsquare}.
	\end{itemize}

	The case where we consider $ i $ and $ 1+i+j+k $ is the same, by looking at \cref{fig:other-P4}.
\end{proof}

Now we have all the tools to prove \cref{res:states-of-p6}.

\begin{proof}[Proof of \cref{res:states-of-p6}]
	Let $F$ be a bad face of $P^6$. If $F=P^6$, then the inherited state is a balanced state of $P^6$ and is totally legal by \cref{res:states-of-p6-are-legal}. If it has codimension $2$, its inherited state is also totally legal by \cref{res:bad-ridge-collapsible}.

	Suppose that $F$ has codimension $3$, so we may assume it is the intersection of $ i $, $ 1+i+j+k $, $ -1+i+j+k $. The dual of $F$ is the usual triangular prism, with vertices corresponding to $ k $, $ -1+i+j-k $, $ j $, $ 1+i-j+k $, $ -1+i-j+k $, $ 1+i+j-k $. Since the state on $P^6$ is balanced, we get that the inherited state is $ \infty $-legal.

	If $F$ has codimension $4$, we may assume, using $ \Rsixteen $ and the symmetry of \cref{fig:P4-link}, that it is the intersection of $ 1+i+j+k $, $ -1+i+j+k $, $j$, and $ -1-i+j+k $. In this case, $F$ is a hyperbolic triangle with two ideal vertices and a finite one, and its dual is the disjoint union of a segment, with vertices labelled by $ k $ and $ 1-i+j+k $, and a point, with label $ -1+i+j-k $. Using the balancedness, we know that the endpoints of the segment have the same inherited status, which is opposite to the status of the point. Therefore, the inherited state is totally legal.

	If $F$ has codimension $6$ it is a vertex, and it is the intersection of three pairs of facets, each pair corresponding to a move. So its dual is a cube that decomposes as a product of three monochromatic squares, as desired.
\end{proof}

\section{The extended cubulation} \label{sec:cusps}

In the previous section we have constructed a Bestvina-Brady Morse function $ f \colon \cubulation \to S^1 $ with well-behaved ascending and descending links, as described in \cref{res:f-regular-or-3-critical}. We would like to promote this to a perfect circle-valued Morse function on $ \mfld $: to do so, we need to be careful as $ \cubulation $ is not homeomorphic to $\mfld$, but is just a spine of $\mfld$.

By truncating $\mfld$ with a small horosphere around each cusp, we obtain a compact manifold $ \truncmfld $, whose interior is diffeomorphic to $\mfld$, and is tessellated into copies of the truncated polytope $ \truncpol $. This compact manifold carries a cube complex structure $ \extcubulation $, as described in \cite[Section 1.6]{IMMHyperbolic5Manifold}, and $\cubulation$ is the subcomplex of $ \extcubulation $ made by all the cubes contained entirely in the interior of $ \truncmfld $.

Let $D$ be a boundary component of $ \extcubulation $. The smallest subcomplex of $ \extcubulation $ that contains all the cubes intersecting $D$ is a collar, and it is naturally isomorphic as a cube complex to $D \times [0,1]$, where $ D \times \set 0 \subseteq \cubulation $, and $ D \times \set1 \subseteq \boundary \extcubulation $.

We extend $ f \colon \cubulation \to S^1 $ to $ \extf \colon \extcubulation \to S^1 $ by setting $ \extf(x,h) = f(x,0) + h $, where $ (x,h) \in D \times [0,1] $.

We subdivide $\extcubulation$ by taking the barycentric subdivision on $\cubulation$, and subdividing each collar $ D \times [0,1] $ as $ (\sd D) \times [0,1] $. Since $f$ is affine on every simplex of $ \cubulation $, then $\extf$ is a Bestvina-Brady map on the subdivision of $\extcubulation$.

\begin{remark}
	This extension produces the same result as the procedure in \cite{IMMHyperbolic5Manifold}. There the edges are always oriented towards the boundary, so the resulting function is as described above on a collar of each boundary component.

	Another possibility was to consider directly the extended cubulation in the previous section, orienting all its edges as done in \cite{IMMHyperbolic5Manifold}. However, by working on $ \cubulation $ first and extending $f$ to $ \extcubulation $ later, we avoid having to perform a full barycentric subdivision of the collars, which are now only subdivided into prisms over simplices of $ \sd C $.
\end{remark}

To prove that the $ \extf $ has ascending and descending links that are either collapsible or collapse to a $2$-sphere, we first need to check the following combinatorial condition.

\begin{lemma}\label{res:fibering-cusp-condition}
	Let $\pol[v]$ be a copy of $ \pol $, and let $c$ be an ideal vertex. There exists a move $m$ such that there are exactly two facets belonging to $m$ and incident to $c$; moreover, these two facets are not adjacent and they have opposite status.
\end{lemma}

\begin{proof}
	Each of the $27$ ideal vertices of $P^6$ is opposed to one of the $27$ facets, in the sense that the ideal vertex opposed to a facet $F$ is incident to the $10$ facets that are not equal nor adjacent to $F$. If we intersect $P^6$ with a small horosphere around an ideal vertex $c$ we obtain a $5$-cube, and in particular the $10$ facets incident to $c$ are partitioned in $5$ pairs of opposite facets, and two facets are adjacent if and only if they are in a different pair.

	By looking at \cref{table:r-of-t}, we now check each cusp and show that for each there is a move that contains exactly two opposite facets incident to the cusp. If those two facets are different from $A$, $B$, and $C$, by the balancedness property of the state on $ \pol[v] $ it follows that they have opposite status, since they are not adjacent.

	\begin{itemize}
		\item The cusp opposed to $A$ is incident to all the facets with label in $Q_8$, and to $B$ and $C$. Every pair of opposite facets, save for $B$ and $C$, satisfy the condition. The same happens for the cusps opposed to $B$ and $C$, since they are incident to $A$ and $C$ (resp.~$A$ and $B$) and to all the facets in the third (resp.~second) column of \cref{table:r-of-t}.
		\item Consider the cusp opposed to $ 1+i+j+k $. The two facets $ -1-i+j-k $ and $ -j $ are incident to it, and all the other facets in the same move are not. Since they are not adjacent to each other, they satisfy the condition. By acting with $ \Rsixteen $, one proves that the condition also holds for all cusps opposed to a facet with label in $ \bintetra \setminus Q_8 $.
		\item The facets $ -1+i+j+k $ and $ -1-i-j-k $ are in the same move, they are not adjacent, and they are incident to the cusp opposed to the facet with label $1$. All other facets in the same move are adjacent to the facet $1$, and therefore not incident to the cusp. Using the action of $Q_8$ we get the condition on all facets with label in $Q_8$. \qedhere
	\end{itemize}
\end{proof}

\begin{proposition}\label{res:fibering-boundary}
	Let $D$ be a boundary component of $\extcubulation$. Then the restriction $ \extf \colon \sd D \to S^1 $ has collapsible ascending and descending links.
\end{proposition}

\begin{proof}

	Let $ \truncpol[v] $ be a copy of $ \truncpol $ that intersects $D$, and let $H$ be the facet contained in $D$. This facet is a Euclidean cube, and it inherits a state and a set of moves from $\pol[v]$. By performing the same procedure described in \cref{sec:bary-sub} on the polytope $H$, we obtain a Bestvina-Brady Morse function on $ \sd D $, which by construction coincides with $f$.

	To prove that all ascending and descending links are collapsible, by \cref{thm:inherited-state-collapsibility} we need to show that the states of all the copies of $H$ are totally legal, and that the same holds for all inherited states of bad faces.

	Note that bad faces of $H$ are precisely the intersection of bad faces of $ \truncpol $ with $H$. The state of $H$ can be obtained by restricting the state of $\truncpol$ to the facets adjacent to $H$, and similarly the inherited state of a face $F$ of $H$ is obtained by restricting the inherited state of the face of $\truncpol$ that intersects $H$ in $F$.

	By \cref{res:fibering-cusp-condition}, the cube $H$ has a pair of opposite facets  that are in the same move, have opposite status, and all other facets are in a different move. Let $F_1$ and $F_2$ be these two facets.

	Let $F$ by any face of $H$. If $F$ is contained in $F_1$, then it is the intersection of $F_1$ and other facets in a different move from $F_1$, so $F$ is good. The same holds if $F$ is contained in $F_2$.

	Otherwise, $F$ intersects both $F_1$ and $F_2$, and $ F \cap F_1 $, $ F \cap F_2 $ have opposite inherited status. The dual of $F$ is an octahedron, and has two opposite vertices with opposite status, so the subcomplexes generated by vertices with status $\Out$ and by vertices with status $\In$ are both cones. So the inherited status is totally legal.
\end{proof}

\begin{proposition}\label{res:perfect-morse-on-extended-cubulation}
	The map $ \extf \colon \extcubulation \to S^1 $ has all ascending and descending links that are either collapsible or they collapse on a $2$-sphere.
\end{proposition}

\begin{proof}
	Let $v$ be a vertex of the subdivision of $ \extcubulation $. If $v \in \boundary\extcubulation$, since $f$ decreases while moving away from the boundary, then $ \linkup[\extf] v $ coincides with $\linkup[\restrict \extf {\boundary \extcubulation}]v$, and we know from \cref{res:fibering-boundary} that it is collapsible. On the other hand, $ \linkdown[\extf]v $ is the cone over $ \linkdown[\restrict \extf {\boundary \extcubulation}]v $, so it is also collapsible.

	Otherwise, let $ v \in \sd \cubulation $. The only vertices whose ascending and descending links in $ \cubulation $ collapse to a $2$-sphere are the barycentres of some $6$-cubes, and they are not contained in any collar, so their ascending and descending link remains the same.

	Otherwise, we may assume that $v$ has collapsible ascending and descending links in $\cubulation$. When adding the collars, we get that $ \linkdown[\extf]v = \linkdown[f]v $, while  $ \linkup[\extf]v$ is obtained from $ \linkup[f]v $ by coning off the intersection of $ \linkup[f]v $ with $ D \times \set0 $, for every boundary component $D$ such that $ D \times [0,1] $ is a collar that contains $v$.

	We claim that each of these subcomplex we are coning off is collapsible. To see this, note that $ D \times \set 0 $ and $ D \times \set1 $ are canonically isomorphic, and $ \restrict f {D \times \set 0} $ coincides with $ \restrict \extf {D \times\set1} $. Combining this by \cref{res:fibering-boundary} we get that $ \restrict f {D \times 0} $ has collapsible ascending and descending link, thus proving the claim.

	So $ \linkup[\extf]v $ is obtained from a collapsible complex by coning some collapsible subcomplexes, and is therefore itself collapsible.
\end{proof}

\section{Smoothing}\label{sec:smoothing}

The aim of this section is to promote the Bestvina-Brady Morse function on $ \extcubulation $ to a smooth one.
In the following, we employ the same terminology and definitions given in \cite{RourkeSanderson}.

Let $M$ a PL $n$-manifold with a compatible structure of an affine cell complex. Given a Bestvina-Brady Morse function $ f \colon M \to \RR $, we can define an analogue of the concepts of regular and critical point that are defined for smooth Morse functions.

\begin{definition}
	Let $M$ be an affine cell complex that is a compact PL manifold, and let $ f \colon M \to S^1 $ be a Bestvina-Brady Morse function. We call a vertex $v$ \newterm{regular} if the descending link is collapsible. We call it \newterm{critical of index $i$} if its descending link collapses to a PL $(i-1)$-sphere.
\end{definition}

This analogy can be made precise by the following statement.

\begin{proposition}\label{prop:smoothing}
	Let $ M $ be an affine cell complex that is a compact PL manifold of dimension $ n \leq 6 $, with possibly non-empty boundary, and $ f\colon M \to S^1 $ be a Bestvina-Brady Morse function.

	Suppose that every vertex $v$ is either regular, or is critical of index $i_v$ and is contained in the interior of $M$. Suppose moreover that the restriction $ \restrict f {\boundary M} $ has only regular vertices.

	Then there exists a unique smooth structure on $M$ compatible with the piecewise-linear one, and there exists a smooth Morse function $ g\colon M \to S^1 $ with the same amount of critical points for every index as $f$.
\end{proposition}

\begin{remark}
	The bound on the dimension is a technical condition under which the smooth and the PL categories coincide. In dimension $7$ there are PL manifolds admitting multiple compatible smooth structures, while in higher dimension there might be no compatible smooth structures. We refer the interested reader to \cite[Section 1.3]{BenedettiSmoothingMorseTheory}.
\end{remark}

The core of the proof of \cref{prop:smoothing} is a lemma of local nature. Before stating it we need to recall the definition of shelling.
\begin{definition}
	An \newterm{elementary shelling} of a PL $n$-manifold $M$ is the removal of $ \interior B \cup (B \cap \boundary M) $, where $B$ is a PL $n$-ball such that $ B \cap \boundary M $ is a $(n-1)$-ball; we say that $M$ shells to $N \subseteq M$ if $N$ can be obtained by a sequence of elementary shellings. Note that, in contrast to collapses, shelling does preserve PL homomorphism type (so if $M$ shells to $N$, then $M$ and $N$ are PL-homeomorphic). For more details, see \cite[Chapter 3]{RourkeSanderson}.
\end{definition}

\begin{lemma}\label{lemma:local-pl-handle}
	Let $M$ be a PL $n$-manifold with boundary, and let $B$ be a PL $n$-ball. Let $ C \subseteq \partial B,\, C' \subseteq \partial M$ be PL $(n-1)$-submanifolds with (possibly empty) boundary, with a PL homeomorphism $ C \isom C' $; let $M'$ be obtained by attaching $B$ on $M$ along $ C \isom C' $, as in \cref{fig:PL-handle}.

	The following hold.

	\begin{itemize}
		\item If $C$ is collapsible, then $M'$ shells to $M$.
		\item If $C$ collapses to a $PL$ $(i-1)$-sphere, then there exist PL manifolds $M \subseteq X \subset Y \subseteq M'$ such that $X$ shells to $M$, $M'$ shells to $Y$ and $Y$ is obtained by attaching an $i$-handle to $X$.
	\end{itemize}
\end{lemma}

\newboolean{collapsible}
\begin{figure}
	\centering
	\setboolean{collapsible}{true}
	\begin{subfigure}{0.4\textwidth}
		\begin{tikzpicture}[scale=2.5]
	\usetikzlibrary{positioning}
	\usetikzlibrary{calc}
	\usetikzlibrary{backgrounds,fit}
	\usetikzlibrary{patterns}

	\draw
	(-1.3,0) -- (-1,0)
	(1,0) -- (1.3,0)
	(-1,0) --  (-.8,0.7) -- (0,1) -- (.8,.8) -- (1,0)
	(0,0) node[below=0.35] {$M$}
	;

	\ifcollapsible
		\draw
		(0,0.5) node {$B$}
		(0.2,0) node[above] {$C$}
		;
		\draw[dashed] (-1,0)--(1,0);
	\else
		\draw
		(-0.2,0.6) node {$B$}
		(0.7,0) node[above] {$C$}
		(-0.7,0) node[above] {$C$}
		(0.5,0) -- (0,0.5) -- (-0.5,0) -- cycle
		;
		\draw[dashed]
		(-1,0)--(-.5,0)
		(1,0)--(.5,0)
		;
	\fi

\end{tikzpicture}
	\end{subfigure}
	\setboolean{collapsible}{false}
	\begin{subfigure}{0.4\textwidth}
		\begin{tikzpicture}[scale=2.5]
	\usetikzlibrary{positioning}
	\usetikzlibrary{calc}
	\usetikzlibrary{backgrounds,fit}
	\usetikzlibrary{patterns}

	\draw
	(-1.3,0) -- (-1,0)
	(1,0) -- (1.3,0)
	(-1,0) --  (-.8,0.7) -- (0,1) -- (.8,.8) -- (1,0)
	(0,0) node[below=0.35] {$M$}
	;

	\ifcollapsible
		\draw
		(0,0.5) node {$B$}
		(0.2,0) node[above] {$C$}
		;
		\draw[dashed] (-1,0)--(1,0);
	\else
		\draw
		(-0.2,0.6) node {$B$}
		(0.7,0) node[above] {$C$}
		(-0.7,0) node[above] {$C$}
		(0.5,0) -- (0,0.5) -- (-0.5,0) -- cycle
		;
		\draw[dashed]
		(-1,0)--(-.5,0)
		(1,0)--(.5,0)
		;
	\fi

\end{tikzpicture}
	\end{subfigure}
	\caption{Attaching a ball $B$ to a manifold. If the attaching locus $C$ is collapsible, this is a shelling (left); if $C$ instead collapses to an $(i-1)$-sphere, we are attaching an $i$-handle (right, with $i=1$). }
	\label{fig:PL-handle}
\end{figure}

\begin{proof}
	We adapt the proof of \cite[Theorem 2.2]{BenedettiSmoothingMorseTheory}, which is a similar statement for discrete Morse functions, to the PL setting.

	If $C$ is collapsible, then the pair $(B, C)$ is PL homeomorphic to the pair of standard simplices $(\Delta^n,\Delta^{n-1})$, so $M'$ shells to $M$ via an elementary shelling.

	Suppose now that $C$ collapses to a PL sphere $S$. Pick a PL triangulation of $\partial B$, such that $C$ and $S$ are subcomplexes, and cone it to a point $v$ to obtain a triangulation of $B$.
	The cone $B'\coloneq vS$ is a $PL$ $i$-ball, so up to performing some barycentric subdivisions on $B'$ we can choose an $i$-simplex $\sigma$ of $B'$ such that $Z \coloneq B' \setminus \interior \sigma$ collapses on $S$ \cite[Lemma 2.22 and Definition 2.17]{BenedettiSmoothingMorseTheory}. Since $M \cup Z$ collapses on $M$ then we can find regular neighbourhoods $X := \regneigh{M \cup Z}$ and $\regneigh{M}$ such that the first shells to the second \cite[Theorem 3.26]{RourkeSanderson}. Moreover, if $Y$ denotes an appropriate neighbourhood of $M \cup B'$, by \cite[Lemma 2.10]{BenedettiSmoothingMorseTheory}, $Y$ is obtained from $X$ by attaching an $i$-handle.

	It remains to show that $M'$ shells to $Y$, but this is true because $M'$ collapses on $M \cup vC$ (we can collapse every simplex not in $vC$ from its free face on $\partial M'$) which in turns collapses to $M \cup vS$, so again by taking regular neighbourhoods we obtain the shelling.
\end{proof}

\begin{proof}[Proof of \cref{prop:smoothing}]
	Pass to the infinite cyclic cover $ \univcov M $, with the lift $ \lift f \colon \univcov M \to \RR $ of $f$, and consider $ N=\lift f^{-1} ([0,1]) $. We assume that $0$ is not the image of any vertex, and that vertices have distinct images under $\lift f$ (both can be achieved with a small perturbation). We want to study $N$ by using the sublevels of $ \lift f $.

	Let $ 0 < a_1 < \dots < a_k < 1 $ be real numbers such that $ \widetilde f^{-1}((a_i, a_{i+1}]) $ contains exactly a vertex.

	By \cref{prop:bb-deformation-lemma}, the sublevel $ M_{\leq a_{i+1}} $ collapses to $ M_{\leq a_i} $ with the descending link of the vertex coned off; by passing to a regular neighbourhood, we obtain that $ M_{\leq a_{i+1}} $ is PL homeomorphic to $ M_{\leq a_i} $ with a PL ball $B$ attached along some submanifold $C$ of $ f^{-1}(a_i) $ which is a neighbourhood of the descending link. So we obtain that:
	\begin{itemize}
		\item either $v$ is regular, and so $C$ is collapsible;
		\item or $v$ is non-degenerate critical of index $i$, so $C$ collapses to an $(i-1)$-sphere.
	\end{itemize}

	By \cref{lemma:local-pl-handle}, the map $ \widetilde f $ produces a relative PL handle decomposition, that is, we are describing $N$ as $ \widetilde f^{-1}([0, \epsilon]) $ with some handles attached. Indeed, when we cross the sublevel of a critical vertex, we are attaching a PL handle.

	Note that these handles are always contained in the interior of $N$, as there are no critical points on the boundary. Since we are in low dimension $ n \leq 6 $, this handle decomposition can be made smooth, and then it can be turned into a smooth Morse function $ \widetilde{g} \colon  N \to [0,1] $. By gluing back the boundaries $ \widetilde f^{-1}(0)$ and $ \widetilde f^{-1}(1)$ we obtain our desired smooth circle-valued map.
\end{proof}

\def\facets{\mathcal F}
\def\vectorspace{(\ZZ[2])^\facets}
\def\smallvectorspace{(\ZZ[2])^7}

\section{The infinite cover and finiteness properties}
\label{sec:finiteness}

Given the manifold $ \mfld $ equipped with the perfect Morse function $ f \colon \mfld \to S^1 $, we consider the infinite cyclic cover $ \abcover $ associated with $ \ker f_* $. Analogously to \cite[Corollary 5]{BattistaHyperbolic4manifolds} we obtain the following.

\begin{corollary}\label{abcover-betti}
	The infinite cyclic cover $ \abcover $ is a geometrically infinite hyperbolic manifold that is diffeomorphic to $ F \times [0,1] $ with infinitely many $3$-handles attached on both sides, where $F$ is the interior of a $5$-manifold with toric boundary, and has nontrivial $ \pi_2(F) $.

	Moreover, $\abcover$ has fundamental group that is finitely presented but not type $ \finitetype[3] $, and has $ b_3(\abcover) = \infty $.
\end{corollary}

Note that there is a key difference when compared to the $4$-dimensional case. The regular fibres obtained in \cite{BattistaHyperbolic4manifolds} are all hyperbolic $3$-manifolds; this cannot happen in dimension $6$, as they are not aspherical.

\begin{proof}
	Lift $ f \colon \mfld \to S^1 $ to a Morse function $ \lift f \colon \abcover \to \RR $, and let $F$ be the preimage of a regular value $t \in \RR$.

	Since $ \lift f $ has infinitely many critical points, all of index $3$, by Morse theory we can obtain $ \abcover $ from $ F \times [t-\epsilon, t+\epsilon] $ by attaching infinitely many $3$-handles.
	From this, it follows that $b_3(\abcover)$ is infinite.

	Attaching $3$-handles does not change the fundamental group, so $ \pi_1(\abcover) \isom \pi_1(F) $ is finitely presented. Since $ \abcover $ is aspherical, it is a classifying space for $ \pi_1(\abcover) $, so $b_3(\pi_1(\abcover))$ is also infinite. Therefore the fundamental group cannot be of type $ \finitetype[3] $.

	It remains to prove that $ \pi_2(F) $ is nontrivial. If $ \pi_2(F)=0 $, then we could construct a classifying space for $ \pi_1(F) $ by attaching to $F$ cells of dimension $ \geq 4 $, that contradicts $ b_3(\pi_1(F)) = b_3(\abcover)=\infty $.
\end{proof}

Since $ \mfld $ is not compact, its fundamental group is not hyperbolic.
However, it is possible to obtain a hyperbolic group by considering an appropriate \emph{filling} of the cusps, similarly to what was done in \cite[Section 3]{IMMHyperbolic5Manifold}.

By removing a small horoball at each cusp of $ \mfld $, we obtain a compact manifold $ \truncmfld$, with boundary components diffeomorphic to the $5$--torus. By possibly passing to a finite cover, we may assume that each boundary component has systole larger than $2\pi$.

The Morse function $f$ restricts to a fibration on each $5$--torus; we may assume, up to an isotopy, that this fibration has geodesic $4$-tori as fibres. Denote by $ \hat \mfld $ the space obtained by coning every $4$-torus to a point, so that $ f $ extends to a map on $ \hat \mfld $. Let $ \hat F $ be the preimage of a regular value $t$ inside $ \hat \mfld $, $ G\coloneq \pi_1(\hat M) $, and $ H \coloneq \pi_1(\hat F) $.

\begin{proposition}
	The group $G$ is Gromov-hyperbolic and torsion-free, and the subgroup $H$ is finitely presented but not of type $ \finitetype[3] $.
\end{proposition}

This resembles closely the construction of a group of type $ \finitetype[3] $ not $ \finitetype[4] $ in \cite{LlosaPyMartelli}, based on a Morse function defined on a hyperbolic $8$-manifold $M^8$. %
Since the Morse function on $M^6$ is perfect, which is a much stronger condition than the one on $M^8$, the proof here is much simpler.

\begin{proof}
	The filling is a $ 2\pi $--filling, and by \cite{FujiwaraManning} $ \hat M $ admits a negatively curved metric, so its fundamental group is Gromov-hyperbolic and torsion free.

	The group $H$ is finitely presented, since it is the fundamental group of a compact pseudo-manifold. The infinite cyclic cover $X$ of $ \hat M $ associated with $H$ is aspherical, and can be obtained from $ \hat F \times (-\epsilon, \epsilon) $ by attaching infinitely many $3$-handles (note that $f$ is a locally trivial fibration around the cones of the $4$-tori).

	This implies that $ b_3(X) = \infty $, and that $ \pi_1(X) \isom H $, so $X$ is a classifying space for $H$. So the third Betti number of $H$ is infinite, and therefore $H$ is not of type $ \finitetype[3] $.
\end{proof}

\section{Related results and some questions}
\label{sec:comments}
The main inspirations for this paper are the combinatorial game in \cite{JankiewiczNorinWise}, and some of its applications \cites{BattistaHyperbolic4manifolds, IMMAlgebraicFibering, IMMHyperbolic5Manifold}.
A natural direction of investigation is to continue pursuing this line towards higher dimensions.
\begin{question}
	Do the manifolds $M^7$ and $M^8$ in \cite{IMMAlgebraicFibering} admit a perfect circle-valued Morse function?
\end{question}
In the case of $M^7$ such a function would be a fibration. Following the intuition of this paper, one could try to partition the $56$ facets of $P^7$ into $7$ moves with $8$ faces in each, and use balanced states to define a Morse function. %
However, the resulting function is necessarily null-homotopic on some cusp of $P^7$, as it does not satisfy \cref{res:fibering-cusp-condition}, so it cannot be a fibration. The reason is that $P^7$ has $126$ ideal vertices, but there would be at most $ 4 \cdot 4 \cdot 7 = 112 $ pairs of facets in the same move and with opposite status.
The problem relative to the manifold $M^7$ has also been studied in \cite{FisherAlgFibring7} with more algebraic techniques.

\medskip

In \cite{BattistaHyperbolic4manifolds} some regular and singular fibres of perfect Morse functions on hyperbolic $4$--manifolds were computed.
They all are hyperbolic $3$-manifolds, and therefore aspherical.

In \cref{abcover-betti}, we have shown that regular fibres of $6$-dimensional hyperbolic manifolds cannot be aspherical. It would be interesting to know more about the regular and singular fibres of the example we constructed.

\begin{question}
	What are the regular fibres of the map $f\colon M^6 \to S^1$? What are the singular fibres?
\end{question}

Finally, we focus on the following question.

\begin{question}[\cite{BattistaHyperbolic4manifolds}, Question 27]\label{q:virtualperfect}
	Does every hyperbolic manifold have a finite cover that admits a perfect circle-valued Morse function?
\end{question}
Since perfect Morse functions lift to finite covers, \cref{thm:M6} shows that \cref{q:virtualperfect} has positive answer for all manifolds commensurable to $P^6$.

Passing to a finite cover is necessary, as it is not hard to build a $6$-manifold that does not admit a perfect Morse function. In fact, two obvious obstructions are the vanishing of the first Betti number, or the presence of one cusp that does not fibre over $S^1$. Indeed, both orientable $6$--manifolds built in \cite{EverittRatcliffeTschantz} satisfy both obstruction properties; they have $b_1=0$ and also a cusp section with vanishing~$b_1$.
Regarding closed manifolds, one possible obstruction would be the vanishing of $b_1$, however the authors do not know any example of a closed hyperbolic $6$--manifold $M$ with $b_1(M)=0$ (there are explicit examples in dimension $5$ \cite{Chen5mflds} and $7$ \cite{BergeronClozel}).

Lastly, we point out that there have been also some findings that could suggest a negative answer to \cref{q:virtualperfect}, at least in a more generalised context: Avramidi, Okun and Schreve \cite{AvramidiOkunSchreve} constructed a closed, aspherical $7$--manifold with Gromov--hyperbolic fundamental group which does not virtually fibre over $S^1$. %

\clearpage

\bibliography{references}

\end{document}